\documentclass[twoside,leqno,10pt, A4]{amsart}
\usepackage{amsfonts}
\usepackage{amsmath}
\usepackage{amscd}
\usepackage{amssymb}
\usepackage{amsthm}
\usepackage{amsrefs}
\usepackage{latexsym}
\usepackage{mathrsfs}
\usepackage{bbm}
\usepackage{enumerate}
\usepackage{graphicx}

\usepackage{amsfonts}
\usepackage{amsmath}
\usepackage{amscd}
\usepackage{amssymb}
\usepackage{amsthm}
\usepackage{amsrefs}
\usepackage{latexsym}
\usepackage{mathrsfs}
\usepackage{bbm}
\usepackage{amscd}
\usepackage{amssymb}
\usepackage{amsthm}
\usepackage{amsrefs}
\usepackage{latexsym}
\usepackage{mathrsfs}
\usepackage{bbm}
\usepackage{enumerate}
\usepackage{graphicx}
\usepackage{color}
\begin{document}

\newtheorem{theorem}[subsection]{Theorem}
\newtheorem{proposition}[subsection]{Proposition}
\newtheorem{lemma}[subsection]{Lemma}
\newtheorem{corollary}[subsection]{Corollary}
\newtheorem{conjecture}[subsection]{Conjecture}
\newtheorem{prop}[subsection]{Proposition}
\numberwithin{equation}{section}
\newcommand{\mr}{\ensuremath{\mathbb R}}
\newcommand{\mc}{\ensuremath{\mathbb C}}
\newcommand{\dif}{\mathrm{d}}
\newcommand{\intz}{\mathbb{Z}}
\newcommand{\ratq}{\mathbb{Q}}
\newcommand{\natn}{\mathbb{N}}
\newcommand{\comc}{\mathbb{C}}
\newcommand{\rear}{\mathbb{R}}
\newcommand{\prip}{\mathbb{P}}
\newcommand{\uph}{\mathbb{H}}
\newcommand{\fief}{\mathbb{F}}
\newcommand{\majorarc}{\mathfrak{M}}
\newcommand{\minorarc}{\mathfrak{m}}
\newcommand{\sings}{\mathfrak{S}}
\newcommand{\fA}{\ensuremath{\mathfrak A}}
\newcommand{\mn}{\ensuremath{\mathbb N}}
\newcommand{\mq}{\ensuremath{\mathbb Q}}
\newcommand{\half}{\tfrac{1}{2}}
\newcommand{\f}{f\times \chi}
\newcommand{\summ}{\mathop{{\sum}^{\star}}}
\newcommand{\chiq}{\chi \bmod q}
\newcommand{\chidb}{\chi \bmod db}
\newcommand{\chid}{\chi \bmod d}
\newcommand{\sym}{\text{sym}^2}
\newcommand{\hhalf}{\tfrac{1}{2}}
\newcommand{\sumstar}{\sideset{}{^*}\sum}
\newcommand{\sumprime}{\sideset{}{'}\sum}
\newcommand{\sumprimeprime}{\sideset{}{''}\sum}
\newcommand{\sumflat}{\sideset{}{^\flat}\sum}
\newcommand{\shortmod}{\ensuremath{\negthickspace \negthickspace \negthickspace \pmod}}
\newcommand{\V}{V\left(\frac{nm}{q^2}\right)}
\newcommand{\sumi}{\mathop{{\sum}^{\dagger}}}
\newcommand{\mz}{\ensuremath{\mathbb Z}}
\newcommand{\leg}[2]{\left(\frac{#1}{#2}\right)}
\newcommand{\muK}{\mu_{\omega}}
\newcommand{\thalf}{\tfrac12}
\newcommand{\lp}{\left(}
\newcommand{\rp}{\right)}
\newcommand{\Lam}{\Lambda_{[i]}}
\newcommand{\lam}{\lambda}
\def\L{\fracwithdelims}
\def\om{\omega}
\def\pbar{\overline{\psi}}
\def\phis{\phi^*}
\def\lam{\lambda}
\def\lbar{\overline{\lambda}}
\newcommand\Sum{\Cal S}
\def\Lam{\Lambda}
\newcommand{\sumtt}{\underset{(d,2)=1}{{\sum}^*}}
\newcommand{\sumt}{\underset{(d,2)=1}{\sum \nolimits^{*}} \widetilde w\left( \frac dX \right) }

\newcommand{\hf}{\tfrac{1}{2}}
\newcommand{\af}{\mathfrak{a}}
\newcommand{\Wf}{\mathcal{W}}

\newtheorem{mylemma}{Lemma}
\newcommand{\intR}{\int_{-\infty}^{\infty}}

\theoremstyle{plain}
\newtheorem{conj}{Conjecture}
\newtheorem{remark}[subsection]{Remark}

\makeatletter
\def\widebreve{\mathpalette\wide@breve}
\def\wide@breve#1#2{\sbox\z@{$#1#2$}%
     \mathop{\vbox{\m@th\ialign{##\crcr
\kern0.08em\brevefill#1{0.8\wd\z@}\crcr\noalign{\nointerlineskip}%
                    $\hss#1#2\hss$\crcr}}}\limits}
\def\brevefill#1#2{$\m@th\sbox\tw@{$#1($}%
  \hss\resizebox{#2}{\wd\tw@}{\rotatebox[origin=c]{90}{\upshape(}}\hss$}
\makeatletter

\title[Bounds for moments of cubic and quartic Dirichlet $L$-functions]{Bounds for moments of cubic and quartic Dirichlet $L$-functions}

\author{Peng Gao and Liangyi Zhao}

\begin{abstract}
 We study the $2k$-th moment of central values of the family of primitive cubic and quartic Dirichlet $L$-functions.   We establish sharp lower bounds for all real $k \geq 1/2$ unconditionally for the cubic case and under the Lindel\"of hypothesis for the quartic case.  We also establish sharp lower bounds for all real $0 \leq k<1/2$ and sharp upper bounds for all real $k \geq 0$ for both the cubic and quartic cases under the generalized Riemann hypothesis (GRH). As an application of our results, we establish quantitative non-vanishing results for the corresponding $L$-values.
\end{abstract}

\maketitle



\section{Introduction}
\label{sec 1}

The moments of $L$-functions are very important in many arithmetic applications. A classical case is the $2k$-th moment of the Riemann zeta function $\zeta(s)$ on the critical line
\begin{align*}
 M_k(T)=\int\limits^{2T}_{T} \left| \zeta \left( \frac{1}{2}+it \right) \right|^{2k} \dif t.
\end{align*}
  In connection with random matrix theory, J. P. Keating and N. C. Snaith \cite{Keating-Snaith02} conjectured precise formulas for $M_k(T)$ for all real $k \geq 0$. The same formulas were also conjectured by  A. Diaconu, D. Goldfeld and J. Hoffstein \cite{DGH} using multiple Dirichlet series. More precise asymptotic formulas with lower order terms are given in the work of J. B. Conrey, D. W. Farmer, J. P. Keating,
  M. O. Rubinstein and N. C. Snaith in \cite{CFKRS}. \newline

The only currently known asymptotic formulas for $M_k(T)$ are for $k=0$, $1$ and $2$ with $k=1$ due to G. H. Hardy and J. E. Littlewood \cite{H&L} and $k=2$ due to A. E. Ingham \cite{In}.  Other than these cases, sharp lower bounds for $M_k(T)$ of the conjectured order of magnitude were proved when $2k$ is a positive integer by K. Ramachandra \cite{Ramachandra2}, for all positive rational numbers $k$ by D. R. Heath-Brown \cite{H-B81-2}, and for all positive real numbers $k$ by K. Ramachandra \cite{Ramachandra1}.  The results for irrational $k$'s required the assumption of the truth of the Riemann hypothesis (RH). In the other direction, sharp upper bounds for $M_k(T)$ of the conjectured order of magnitude were known unconditionally for $k = 1$ and under RH for $0 <k< 2$ by K. Ramachandra \cite{Ramachandra3}. The ranges of validity of the upper bounds were extended to  $k=1/n$ for positive integers $n$ unconditionally and $0 <k \leq 2$ under RH by D. R. Heath-Brown \cite{H-B81-2}, and were further extended to $0<k<2+2/11$ by M. Radziwi{\l\l} \cite{Radziwill} under RH and to $k = 1 + 1/n$ for positive integers $n$ by S. Bettin, V. Chandee and M. Radziwi{\l\l}  \cite{BCR}. \newline

    In \cite{R&Sound, R&Sound1}, Z. Rudnick and K. Soundararajan developed a simple and powerful method towards establishing sharp lower bounds for moments of families of $L$-functions and this method was extended by M. Radziwi{\l\l} and K. Soundararajan \cite{Radziwill&Sound} to obtain the desired lower bounds for $M_k(T)$  for any real number $k > 1$ unconditionally. In \cite{Sound2009}, K. Soundararajan introduced a method that allows one to essentially derive sharp upper bounds for moments of families of $L$-functions under the generalized Riemann hypothesis (GRH). A refinement of this method by A. J. Harper \cite{Harper} led to the desired upper bounds for  $M_k(T)$ for all $k \geq 0$. \newline

    In \cite{Radziwill&Sound}, M. Radziwi{\l\l}  and K. Soundararajan developed an upper bounds principle to study moments of families of $L$-functions unconditionally and applied the method for the family of quadratic twists of $L$-functions associated with elliptic curves.  This modus operandi was carried out further by W. Heap, M. Radziwi{\l\l} and K. Soundararajan in \cite{HRS} to establish sharp upper bounds for $M_k(T)$ for $0 \leq k \leq 2$ unconditionally.  A dual principle was developed by W. Heap and K. Soundararajan in \cite{H&Sound} to prove sharp lower bounds for $M_k(T)$ for all real $k \geq 0$ unconditionally. \newline

  As both of the above principles of M. Radziwi{\l\l} and K. Soundararajan and of W. Heap and K. Soundararajan work for general families of $L$-functions, they can be applied to study many important families of $L$-functions, beyond the prototypical $\zeta(s)$.  For example, the first-named author applied them to in the study of the bounds for moments of central values of the family of quadratic Dirichlet $L$-functions in \cites{Gao2021-2, Gao2021-3}. \newline

  As Dirichlet characters of a fixed order have significant applications in number theory, it is investigate families $L$-functions attached to these characters. In this paper, we aim to study moments of central values of families of $L$-functions associated with either primitive cubic or quartic Dirichlet characters. Previously, the first moments of these families are obtained in the work of S. Baier and M. P. Young \cite{B&Y} for the cubic case and of the authors \cite{G&Zhao7} for the quartic case. The result in \cite{G&Zhao7} is obtained under the Lindel\"of hypothesis. \newline

We further note that, according to the density conjecture of N. Katz and P. Sarnak \cite{K&S} on the low-lying zeros of families of $L$-functions, the underlying symmetries for the family of quadratic Dirichlet $L$-functions are not the sam as those attached to Dirichlet characters of a fixed higher order. Indeed, the family of quadratic Dirichlet $L$-functions is a symplectic family and  those of cubic and quartic Dirichlet $L$-functions are both unitary families(see \cite{G&Zhao2}).  Thus, the moments that we study herein should resemble those of the Riemann zeta-function on the critical line.  We demonstrate this similarity in the paper by establishing sharp upper and lower bounds for these moments. \newline

For lower bounds, we shall apply the lower bounds principle of W. Heap and K. Soundararajan \cite{H&Sound} to our setting. For the case $k \geq 1/2$, the results depend essentially on evaluations of twisted first moments of cubic and quartic Dirichlet $L$-functions. To state our results, we first introduce some notations.  We write $K$ for either the number field $\mq(i)$ or $\mq(\omega)$ (where $\omega = \exp (2 \pi i/3)$) and $\zeta_K(s)$ for the corresponding Dedekind zeta function.  Let $N(n)$ stand for the norm of any $n \in K$ and let $r_K$ be the residue of $\zeta_K(s)$ at $s = 1$.  We also use $D_K$ to denote the discriminant of $K$ and we recall that (see \cite[sec 3.8]{iwakow}) $D_{\mq(\omega)}=-3, D_{\mq(i)}=-4$.  We reserve the letter $p$ for a prime number in $\mz$ and the letter $\varpi$ for a prime in $K$. For any integer $c \in \mz$, define
\begin{align}
\label{gc}
  g(c)=& \prod_{\varpi | c} (1+ N(\varpi)^{-1})^{-1} \prod_{p | c} \Big(1-\frac 1{p^2}\prod_{\varpi | p} (1- N(\varpi)^{-2})^{-1}  \Big )^{-1}.
\end{align}
We henceforth use the usual convention that an empty product is defined to be $1$.  The same notation $g(c)$ will be used for both $K=\mq(\omega)$ and $\mq(i)$. Thus the meaning of $\varpi$ may vary accordingly.  The distinction should be clear from the context. \newline

  We also define for any integer $\ell \in \mz$,
\begin{equation}
\label{zk}
 c_{K}= r_K \zeta^{-1}_{K}(2) \prod_{(p, D_K)=1}\Big (1-\frac 1{p^2} \prod_{\varpi | p} (1- N(\varpi)^{-2})^{-1} \Big ) \quad \mbox{and} \quad
 Z_K(u, \ell) =  \sum_{m=1}^{\infty} m^{-u} g\Big (\frac {m}{(m,|D_K|\ell)}\Big ),
\end{equation}
   where again $\varpi$ are primes in the corresponding number field $K$. \newline

Let $\Phi$ for a smooth, non-negative function compactly supported on $[1, 2]$ satisfying $\Phi(x) \leq 1$ for all $x$ and $\Phi(x) =1$
for $x\in [3/2,5/2]$, and define, for any complex number $s$,
\begin{equation*}
{\widehat \Phi}(s) = \int\limits_{0}^{\infty} \Phi(x)x^{s}\frac {\dif x}{x}.
\end{equation*}

  Our approach to the lower bounds needs the following result on the twisted first moments of cubic and quartic Dirichlet $L$-functions.
\begin{theorem}
\label{twistedfirstmoment}
  With the notations above, let $X$ be a large real number and $\ell$ a fixed positive integer.  Write $\ell$ uniquely as $\ell=\ell_1\ell^2_2\ell^3_3$ with $\ell_1, \ell_2$ square-free and $(\ell_1, \ell_2)=1$.  We have
\begin{align}
\label{eq:1}
 \sum_{(q,3)=1}\;  \sumstar_{\substack{\chi \shortmod{q} \\ \chi^3 = \chi_0}} L \left( \frac{1}{2}, \chi \right) \chi(\ell) \Phi\leg{q}{X} = c_{\mq(\omega)}  g(3\ell) X \frac 1{\sqrt{\ell^2_1\ell_2}}\widehat{\Phi}(1)Z_{\mq(\omega)} \left( \frac32, \ell \right) + O \left( X^{37/38 + \varepsilon}\ell^{2/3+\varepsilon} \right),
\end{align}
  where the asterisk on the sum over $\chi$ restricts the sum to primitive characters and $\chi_0$ denotes the principal character. \newline

  If we write $\ell$ uniquely as $\ell=\ell_1\ell^2_2\ell^3_3\ell^4_4$ with $\ell_1, \ell_2, \ell_3$ square-free, pair-wise coprime and assume the truth of the Lindel\"of hypothesis, then
\begin{align}
\label{eq:2}
 \sum_{(q,2)=1}\;  \sumstar_{\substack{\chi \shortmod{q} \\ \chi^4 = \chi_0}} L \left( \frac{1}{2}, \chi \right) \chi(\ell) \Phi\leg{q}{X} =c_{\mq(i)}  g(2\ell) X \frac 1{\sqrt{\ell^3_1\ell^2_2\ell_3}}\widehat{\Phi}(1)Z_{\mq(i)}(2, \ell) + O\left( X^{9/10 + \varepsilon}\ell^{1/4 + \varepsilon} \right),
\end{align}
where the asterisk on the sum over $\chi$ restricts the sum to primitive characters $\chi$ such that $\chi^2$ remains primitive.
\end{theorem}

  With the aid of Theorem \ref{twistedfirstmoment}, we establish the following lower bounds for the families of $L$-functions under our consideration.
\begin{theorem}
\label{thmlowerbound}
With the notations above and the truth of the Lindel\"of hypothesis for Dirichlet $L$-functions associated with primitive quartic Dirichlet characters, we have, for large $X$ and all real numbers $k \geq 1/2$,
\begin{align} \label{lowerbounds}
  \sum_{\substack{(q,3)=1 \\ q \leq X}}\;  \sumstar_{\substack{\chi \shortmod{q} \\ \chi^3 = \chi_0}} \left| L \left( \frac{1}{2}, \chi \right) \right|^{2k}  \gg_k  X(\log X)^{k^2} \quad \mbox{and} \quad  \sum_{\substack{(q,2)=1 \\ q \leq X}}\;  \sumstar_{\substack{\chi \shortmod{q} \\ \chi^4 = \chi_0}} \left| L \left( \frac{1}{2}, \chi \right) \right|^{2k}  \gg_k  X(\log X)^{k^2}.
\end{align}
\end{theorem}

For the case $0 \leq k <1/2$, the lower bounds principle requires knowledge on the twisted second moments of cubic and quartic Dirichlet $L$-functions and the same requirement is needed in the upper bounds principle of M. Radziwi{\l\l}  and K. Soundararajan \cite{Radziwill&Sound}. As an unconditional result on the twisted second moments is not currently known, we apply the method of Soundararajan in \cite{Sound2009} as well as its refinement by Harper in \cite{Harper} instead, obtaining some conditional upper bounds as follows.
\begin{theorem}
\label{thmupperbound}
   With the notations above and the truth of GRH, we have, for large $X$ and all real numbers $k \geq 0$,
\[  \sum_{\substack{(q,3)=1 \\ q \leq X}}\;  \sumstar_{\substack{\chi \shortmod{q} \\ \chi^3 = \chi_0}} \left| L \left( \frac{1}{2}, \chi \right) \right|^{2k}  \ll_k X(\log X)^{k^2} \quad \mbox{and} \quad \sum_{\substack{(q,2)=1 \\ q \leq X}}\;  \sumstar_{\substack{\chi \shortmod{q} \\ \chi^4 = \chi_0}} \left| L \left( \frac{1}{2}, \chi \right) \right|^{2k}  \ll_k  X(\log X)^{k^2} . \]
\end{theorem}

We note here that the case $k=1$ in Theorem \ref{thmupperbound} improves the known results given in \cite[Theorem 1.3]{B&Y} and \cite[Theorem 1.3]{G&Zhao7} under GRH.  The above-mentioned approaches of Soundararajan \cite{Sound2009} and Harper \cite{Harper} also enable us to evaluate the twisted second moments under GRH. This, together with the lower bounds principle, allows us to extend the results in Theorem \ref{thmlowerbound} to the case $0 \leq k < 1/2$ conditionally.
\begin{theorem}
\label{thmlowerbound1}
    The bounds given in \eqref{lowerbounds} hold for $0 \leq k < 1/2$ under GRH.
\end{theorem}

Combining Theorems \ref{thmlowerbound}--\ref{thmlowerbound1}, we readily deduce the following result concerning the order of magnitude of the $2k$-th moment of the family of $L$-functions of our interest.
\begin{theorem}
\label{thmorderofmag}
   With the notations above and the truth of GRH, we have, for large $X$ and all real numbers $k \geq 0$,
\[   \sum_{\substack{(q,3)=1 \\ q \leq X}}\;  \sumstar_{\substack{\chi \shortmod{q} \\ \chi^3 = \chi_0}} \left| L \left( \frac{1}{2}, \chi \right) \right|^{2k}   \asymp_k  X(\log X)^{k^2}  \quad \mbox{and} \quad \sum_{\substack{(q,2)=1 \\ q \leq X}}\;  \sumstar_{\substack{\chi \shortmod{q} \\ \chi^4 = \chi_0}} \left| L \left( \frac{1}{2}, \chi \right) \right|^{2k}   \asymp_k  X(\log X)^{k^2}. \]
\end{theorem}

We note here that one can readily deduce from the above theorem that the number of primitive cubic or quartic Dirichlet characters $\chi$ with conductor $\leq X$ such that the corresponding $L(1/2, \chi) \neq 0$ is $\gg X/\log X$ under GRH, via standard arguments as in the proof of \cite[Corollary 1.2]{B&Y}. In fact, by incorporating the above mentioned Harper's method in \cite{Harper}, one may further compute the mollified second moment of the $L$-functions under consideration to obtain a positive proportion of non-vanishing result. This approach was used by S. Lester and M. Radziwi{\l\l} in \cite{LR21} for the mollified moments of quadratic twists of modular $L$-functions and in the function field setting by C. David, A. Florea and M. Lalin \cite{DFL21} to establish a positive proportion non-vanishing result of cubic $L$-functions. In \cite{DFL21}, it was also asserted that the methods can be used to give a positive proportion nonvanishing result in the number field setting on GRH. \newline

Motivated by the work in \cite{LR21} and \cite{DFL21}, we end the introduction by giving the following theorem which states that positive proportions of the members in both of the families of $L$-functions associated with cubic and quartic Dirichlet characters do not vanish at the central point.

\begin{theorem}
\label{coro:nonvanish}
Assume the truth of GRH.  There exist infinitely many primitive Dirichlet characters $\chi$ of order $3$ and $4$ such that $L(1/2, \chi) \neq 0$.  More precisely, the number of such characters with conductor $\leq X$ is $\gg X$.
\end{theorem}

\section{Preliminaries}
\label{sec 2}

   In this section, we gather several auxiliary results required in the course of our proofs.
\subsection{Sums over primes}
\label{sec2.4}

  We first note the following result on various sums over prime numbers.
\begin{lemma}
\label{RS} Let $x \geq 2$. We have, for some constant $b$,
\begin{equation} \label{merten}
\sum_{p\le x} \frac{1}{p} = \log \log x + b+ O\Big(\frac{1}{\log x}\Big).
\end{equation}
 Also, for any integer $j \geq 1$, we have
\begin{equation} \label{mertenpartialsummation}
\sum_{p\le x} \frac {(\log p)^j}{p} = \frac {(\log x)^j}{j} + O((\log x)^{j-1}).
\end{equation}
 Let $\chi$ be a primitive Dirichlet character modulo $q$ and assume that GRH hold for $L(s, \chi)$, we have
\begin{align}
\label{PIT}
 \sum_{p \leq x }\log p \cdot \chi(p) =\delta_{\chi=\chi_0}x+O(\sqrt{x} \left(\log 2qx)^2 \right),
\end{align}
 where we define $\delta_{\chi=\chi_0}=1$ if $\chi=\chi_0$ and $\delta_{\chi=\chi_0}=0$ otherwise.
\end{lemma}

\begin{proof}
The formula \eqref{merten} is a well-known formula due to Mertens (see \cite[Theorem 2.7]{MVa1}.) and \eqref{mertenpartialsummation} follows from \eqref{merten} by partial summation.  \eqref{PIT} is given in \cite[Theorem 5.15]{iwakow}.
\end{proof}

\subsection{Cubic and quartic Dirichlet characters}
    Recall that we write $K$ for either $\mq(\omega)$ or $\mq(i)$. We further use $\mathcal{O}_K$ to denote the ring of integers in $K$ and $U_K$ the group of units in $\mathcal{O}_K$.  It is well-known that $K$ has class number one with $\mathcal{O}_{K}=\mz[\omega]$ or $\mz[i]$.  Recall also that every ideal in $\intz[\omega]$ co-prime to $3$ has a unique generator congruent to $1$ modulo $3$
(see \cite[Proposition 8.1.4]{BEW}) and every ideal in $\intz[i]$ coprime to $2$ has a unique generator congruent to $1$ modulo $(1+i)^3$
(see the paragraph above Lemma 8.2.1 in \cite{BEW}). These generators are called primary. \newline

 For $K=\mq(\omega)$, the cubic residue symbol $\leg{\cdot}{\varpi}_3$ is defined for any prime $\varpi$ co-prime to $3$ in $\mathcal{O}_{K}$, such that
 we have $\leg{a}{\varpi}_3 \equiv a^{(N(\varpi)-1)/3} \pmod{\varpi}$ with $\leg{a}{\varpi}_3 \in \{ 1, \omega, \omega^2 \}$ for any $a \in \mathcal{O}_{K}$, $(a, \varpi)=1$. We also define $\leg{a}{\varpi}_3 =0$ if $\varpi | a$. The definition of the cubic symbol is then extended multiplicatively to $\leg{\cdot}{n}_3$ for any composite $n$ with $(N(n), 3)=1$.
In like manner, for $K=\mq(i)$, the quartic residue symbol $\leg{\cdot}{\varpi}_4$ is defined for any prime $\varpi$ co-prime to $2$
in $\mathcal{O}_{K}$ by $\leg{a}{\varpi}_4 \equiv
a^{(N(\varpi)-1)/4} \pmod{\varpi}$ with $\leg{a}{\varpi}_4 \in \{ \pm 1, \pm i \}$ for any $a \in \mathcal{O}_{K}$, $(a, \varpi)=1$ and $\leg{a}{\varpi}_4 =0$ when $\varpi | a$. Thus the quartic symbol $\leg{\cdot}{n}_4$ can be defined any composite $n$ with $(N(n), 2)=1$, extending $\leg{\cdot}{\varpi}_4$ multiplicatively.  Naturally, we set $\leg{\cdot }{n}_3=\leg{\cdot }{n}_4=1$ for $n \in U_{K}$. \newline

Combining the statements of \cite[Lemma 2.1]{B&Y} and \cite[Lemma 2.1]{G&Zhao7}, we have the following description of primitive cubic and quartic Dirichlet characters.
\begin{lemma}
\label{lemma:cubicclass}
 The primitive cubic Dirichlet characters of conductor $q$ coprime to $3$ are of the form $\chi_n:m \rightarrow \leg{m}{n}_3$ for some $n \in \mz[\omega]$, $n \equiv 1 \pmod{3}$, $n$ square-free and not divisible by any rational primes, with norm $N(n) = q$.  The primitive quartic Dirichlet characters of conductor $q$ coprime to $2$ such that their squares remain primitive are of the form $\chi_n:m \mapsto \leg{m}{n}_4$ for some $n \in \mz[i]$, $n \equiv 1 \pmod{(1+i)^3}$, $n$ square-free and not divisible by any rational primes, with norm $N(n) = q$.
\end{lemma}

  We reserve $\psi_{m}$ for the Hecke characters in $K$ such that $\psi_{m}((n)) = \leg{m}{n}_3$ for $n \in \mq(\omega)$ coprime to $3$
   or $\psi_{m}((n)) = \leg{m}{n}_4$ for $n \in \mq(i)$ coprime to $2$.
 It is shown in \cite[Section 2.1]{B&Y} and \cite[Section 2.1]{G&Zhao7} that $\psi_m$ is either a cubic  Hecke character of trivial infinite type
  modulo $9m$ or a
  quartic Hecke character of trivial infinite type modulo $16m$. We define $\delta_{n=\text{cubic}}$ to be $1$ or $0$ depending on whether $n$ equals a cube or not, and we define $\delta_{n=\text{fourth power}}$ similarly.  Similar to \cite[Proposition 1]{Radziwill&Sound}, we need estimations on smoothed sums of cubic and quartic characters in this paper. Our next result is analogue to \cite[Lemma 2]{H&P}, which concerned with smoothed version of the classical P\'olya inequality over number fields.
\begin{lemma}
\label{PropDirpoly}  With the notations above, for large $X$ and any positive integer $c$, we have
\begin{align}
\label{cubicandquarticcharsum}
\begin{split}
\sum_{(q,3)=1} \ \sumstar_{\substack{\chi \shortmod{q} \\ \chi^3 = \chi_0}} \chi(c) \Phi\Big(\frac{q}{X}\Big)=&
\displaystyle \delta_{c=\text{cubic}}c_{\mq(\omega)} {\widehat \Phi}(1) X g(3c) + O( X^{1/2+\varepsilon}c^{1/2+\varepsilon} ), \\
\sum_{(q,2)=1} \ \sumstar_{\substack{\chi \shortmod{q} \\ \chi^4 = \chi_0}} \chi(c) \Phi\Big(\frac{q}{X}\Big)=&
\displaystyle \delta_{c=\text{fourth power}}c_{\mq(i)} {\widehat \Phi}(1) X g(2c) + O( X^{1/2+\varepsilon}c^{1/2+\varepsilon} ).
\end{split}
\end{align}
Here $c_{\mq(\omega)}$ and $c_{\mq(i)}$ are defined in \eqref{zk} and $g$ in \eqref{gc}.
\end{lemma}
\begin{proof}
  As both cases are similar, we shall only prove the first expression in \eqref{cubicandquarticcharsum}.  Lemma \ref{lemma:cubicclass} gives
\begin{align*}
\begin{split}
 CS:=\sum_{(q,3)=1} \ \sumstar_{\substack{\chi \shortmod{q} \\ \chi^3 = \chi_0}} \chi(c) \Phi\Big(\frac{q}{X}\Big)
=
\sumprime_{\substack{ n \equiv 1 \shortmod{3} }}\chi_n(c)\Phi\left(\frac{N(n)}{X}\right),
\end{split}
\end{align*}
where $\Sigma^{'}$ indicates that the sum runs over square-free elements $n$ of $\mathbb{Z}[\omega]$ with no rational prime divisor. \newline

Let $\mu_{\omega}(l)$ be the M\"{o}bius function on $\mz[\omega]$ and $\mu_{\mz}(d) = \mu(|d|)$ for $d \in \mz$,
where $\mu$ stands for the usual M\"{o}bius function.  Consider the sum
\[ \sum_{\substack{d|n, d \in \intz \\ d \equiv 1 \bmod{3}}} \mu_{\intz} (d) . \]
This sum (see \cite[(21)]{B&Y}) is 1 if $n$ has no rational prime divisor and 0 otherwise.  We thus use it to detect the summation condition on $CS$.  Re-writing $n$ as $dn$.  If $d$ is square-free, then so is $dn$ if and only if $n$ is square-free and prime to $d$.  Thus,
\[ CS= \sum_{\substack{d \in \mz \\ d \equiv 1 \shortmod{3} }} \mu_{\mz}(d)\leg{c}{d}_3  \sum_{\substack{n \equiv 1 \shortmod{3} \\ n \; \mbox{\scriptsize square-free} \\ (n,d) = 1}} \leg{c}{n}_3 \Phi\left(\frac{N(nd)}{X}\right). \]
Now we detect the condition that $n$ is square-free using $\mu_{\omega}$ to arrive at
\[ CS= \sum_{\substack{d \in \mz \\ d \equiv 1 \shortmod{3} }} \mu_{\mz}(d) \sum_{\substack{l \equiv 1 \shortmod{3} \\ (l, d)=1}} \mu_{\omega}(l) \leg{c}{d l^2}_3  \sum_{\substack{n \equiv 1 \shortmod{3} \\ (n,d) = 1}} \leg{c}{n}_3 \Phi\left(\frac{N(ndl^2)}{X}\right). \]

  We evaluate the last sum above by applying Mellin inversion to obtain that
\begin{equation*}
  \sum_{\substack{n \equiv 1 \shortmod{3} \\ (n,d) = 1}} \leg{c}{n}_3 \Phi\left(\frac{N(ndl^2)}{X}\right)=\frac{1}{2 \pi i} \int\limits_{(2)} \leg{X}{N(dl^2)}^s L(s, \psi_{c}) \widehat{\Phi}(s) \dif s.
\end{equation*}

  Note that integration by parts shows that $\widehat{\Phi}(s)$ is a function satisfying the bound for all $\Re(s) > 0$, and integers $E>0$,
\begin{align}
\label{boundsforphi}
  \widehat{\Phi}(s) \ll \min (1, |s|^{-1}(1+|s|)^{-E}).
\end{align}

  We then move the contour of the integral above to $\Re(s)=1/2$ and apply \eqref{boundsforphi} to deduce that the integral on the new line is
\begin{align*}
 \ll X^{1/2} \sum_{d \ll \sqrt{X}} \sum_{N(l) \ll \sqrt{X}} \frac{1}{\sqrt{N(dl^2)}} \int\limits_{-\infty}^{\infty} \left| L \left( \frac{1}{2} + it, \psi_{c} \right) \widehat{\Phi} \left( \frac{1}{2} + it \right) \right| \dif t \ll X^{1/2+\varepsilon}c^{1/2+\varepsilon},
\end{align*}
 where the last estimation above follows from the convexity bound for $L(s, \psi_{c})$ (see \cite[(5.20)]{iwakow}), which asserts that for $0 \leq \sigma \leq 1$,
\begin{equation*}
 |L(\sigma + it, \psi_{c})| \ll (N(c) (1+|t|^2))^{(1-\sigma)/2+\varepsilon},
\end{equation*}
  since the Hecke $L$-function $L(s, \psi_{c})$ has conductor $\ll N(c)|s|^2$.  Also recall that $N(c)=c^2$ if $c \in \intz$. \newline

  We encounter a pole at $s=1$ in the above process only when $c$ is a cube and the contribution to $CS$ of this residue equals
\begin{equation} \label{reseval}
X \widehat{\Phi}(1) \text{Res}_{s=1} L(s, \psi_{c})  \sum_{\substack{d \in \mz, (c,d)=1 \\ d \equiv 1 \shortmod{3} }} \frac{\mu_{\mz}(d)}{d^2} \sum_{\substack{l \equiv 1 \shortmod{3} \\ (l, cd)=1}} \frac{\mu_{\omega}(l)}{N(l^2)} .
\end{equation}
Now if $c$ is a cube, then
\[ L(s, \psi_c) = \zeta_{\ratq(\omega)}(s) \prod_{\varpi|3c} (1-N(\varpi)^{-s}) . \]
Using this and evaluating the sum over $l$, \eqref{reseval} can be recast as
\[ X \widehat{\Phi}(1) r_{\ratq(\omega)} \zeta_{\ratq(\omega)}^{-1}(2) \prod_{\varpi|3c}  (1+N(\varpi)^{-1})^{-1} \sum_{\substack{d \in \mz, (c,d)=1 \\ d \equiv 1 \shortmod{3} }} \frac{\mu_{\mz}(d)}{d^2} \prod_{\varpi|d} (1-N(\varpi)^{-2})^{-1} .  \]

  We then evaluate the sums above to arrive at the first expression in \eqref{cubicandquarticcharsum} and this completes the proof.
\end{proof}

\subsection{The approximate functional equation}
\label{sec: afe}

   Let $\chi$ be any primitive Dirichlet character modulo $q$ and $\af=0$ or $1$ be given by $\chi(-1)=(-1)^{\af}$. We define
\begin{align*}
  \Lambda(s, \chi)= \left( \frac {\pi}{q} \right)^{-(s+\af)/2}\Gamma \left( \frac 12(s+\af) \right)L(s, \chi).
\end{align*}
   Then $\Lambda(s, \chi)$ extends to an entire function on $\mc$ when $\chi \neq \chi_0$ and satisfies the functional equation (see \cite[Theorem 4.15]{iwakow})
\begin{align*}
  \Lambda(1-s, \overline \chi)=\frac {i^{\af}q^{1/2}}{\tau(\chi)} \Lambda(s, \chi).
\end{align*}

Let $G(s)$ be any even function which is holomorphic and bounded in the strip $-4<\Re(s)<4$ satisfying $G(0)=1$. From \cite[Theorem 5.3]{iwakow}, we have the following approximate functional equation for Dirichlet $L$-functions.
\begin{prop} \label{prop:AFE}
Suppose $\chi$ be a primitive Dirichlet character modulo $q$. Let $A$ and $B$ be positive real numbers such that $AB = q$.  Then we have
\begin{align*}
L \left( \frac{1}{2} , \chi \right) = \sum_{m=1}^{\infty} \frac{\chi(m)}{m^{1/2 }} V_{\af}\left(\frac{m}{A}\right)
+ \epsilon(\chi) \sum_{m=1}^{\infty} \frac{\overline{\chi}(m)}{m^{1/2}} V_{\af}\left(\frac{m}{B}\right),
\end{align*}
  where
\begin{align*}
\epsilon(\chi) = i^{-\af} q^{-1/2} \tau(\chi), \quad V_{\af}(x) = \frac{1}{2\pi i} \int\limits_{(2)} \frac{G(s)}{s} \gamma_{\af}(s) x^{-s} \dif s,  \quad \gamma_{\af}(s) = \pi^{-s/2} \frac{\Gamma\left(\tfrac{1/2 + \af+ s}{2}\right)}{\Gamma\left(\tfrac{1/2 + \af}{2}\right)}.
\end{align*}
\end{prop}

\subsection{Upper bound for $\log |L(1/2, \chi)|$ }

  Let $\Lambda(n)$ be the von Mangoldt function on $\mz$.  The following lemma provides an upper bound of $\log |L(1/2, \chi)|$ in terms of a sum involving prime powers.
\begin{lemma}
\label{lem: logLbound}
Let $\chi$ be a non-principal primitive Dirichlet character modulo $q$. Assume the truth of GRH for $\zeta(s)$ and for $L(s, \chi)$. Let $x \geq 2$ and $\lambda_0=0.4912\ldots$ be the unique positive real number satisfying $e^{-\lambda_0} = \lambda_0+ \lam^2_0/2$.
We have for $\lambda \geq \lambda_0$,
\begin{align}
\label{logLupperbound}
\log |L(1/2, \chi)| \le \Re{\sum_{\substack{2 \leq n \le x }} \frac{\Lam(n)\chi(n)}{n^{1/2+ \lam/\log x} \log n}
\frac{\log (x/n)}{\log x}}+\frac {\log q}{2} \left( \frac{1+\lam}{\log x} \right)+ O\Big( \frac{1}{\log x}\Big).
\end{align}
\end{lemma}

\begin{proof}
This lemma here as it can be established along similar lines as to \cite[Proposition]{Sound2009}(see also \cite[Section 4]{Sound2009} for the treatment of families of $L$-functions).
\end{proof}

Our next lemma treats essentially the sum over prime squares in \eqref{logLupperbound}.  The proof follows that of \cite[Lemma 2]{Sound2009}.
\begin{lemma}
\label{lem: primesquareLbound}
Let $\chi$ be a non-principal primitive Dirichlet character modulo $q$ whose square remains primitive modulo $q$.  Assume GRH for $L(s, \chi)$ and use the same notations as in Lemma \ref{lem: logLbound}.  We have, for $x \geq 2$ and $q \leq X$ for a large number $X$,
\begin{align*}
 \sum_{\substack{ p \leq x^{1/2} }}  \frac{\chi(p^2)}{p^{1+2\lambda/\log x}}  \frac{\log (x/p^2)}{\log x}
 =\sum_{\substack{ p \leq \min (x^{1/2}, \log X) }} \frac{\chi(p^2)}{p^{1+2\lambda/\log x}}  \frac{\log (x/p^2)}{\log x}+O(1)
 =O(\log \log \log X).
\end{align*}
\end{lemma}
\begin{proof}
   We may assume that $\log X \leq x^{1/2}$. As $\chi(p^2)=\chi^2(p)$ and $\chi^2$ is also primitive modulo $q \leq X$ by our assumption,  we apply \eqref{PIT} and get
\begin{align*}
 \sum_{p \leq y }\log p \cdot \chi(p^2) = O(\sqrt{y} \left(\log 2xy)^2 \right).
\end{align*}
   The above estimation, together with partial summation, yields
\begin{align*}
 \sum_{\substack{(\log X)^6 < p \leq x^{1/2} }}  \frac{\chi(p^2)}{p^{1+2\lambda/\log x}}   = O(1).
\end{align*}
Now Lemma \ref{RS} gives
\begin{align*}
  \sum_{\substack{p  \leq (\log X)^6}}  \frac{\chi(p^2)}{p^{1+2\lambda/\log x}} \ll \sum_{\substack{p \leq (\log X)^6}}
  \frac{1}{p}  =  O(\log \log \log X).
\end{align*}
   Applying Lemma \ref{RS} one more time, we arrive at
\begin{align*}
  \sum_{\substack{ p \leq x^{1/2} }}  \frac{\chi(p^2)}{p^{1+2\lambda/\log x}}  \frac{\log p}{\log x} \ll  \frac 1{\log x} \sum_{\substack{ p \leq x^{1/2} }}  \frac{\log p}{p} =O(1).
\end{align*}
Now the lemma readily follows from the above estimates.
\end{proof}

   Observe further that Lemma \ref{RS} implies that the terms on the right side of \eqref{logLupperbound} corresponding to $n=p^l$ with $l \geq 3$
    contribute $O(1)$.  We apply this observation together with Lemma \ref{lem: logLbound} and Lemma \ref{lem: primesquareLbound} by taking $\lambda=\lambda_0$ and $\lambda=1$ to arrive at the following upper bounds for $\log |L(1/2, \chi)|$ involving with cubic and quartic Dirichlet characters.
\begin{lemma}
\label{lem: logLbound1}
Let $\chi$ be a non-principal primitive cubic or quartic Dirichlet character modulo $q$ and suppose that $\chi^2$ remains primitive modulo $q$.
Assume GRH for $L(s, \chi)$ and use the same notations as in Lemma \ref{lem: logLbound}. We have, for $x \geq 2$ and $q \leq X$ for a large number $X$,
\begin{align}
\label{equ:3.3}
\begin{split}
 & \log  |L(1/2, \chi)| \leq \Re \left( \sum_{\substack{  p \leq x }} \frac{\chi (p)}{p^{1/2+\lambda_0/\log x}}  \frac{\log (x/p)}{\log x} \right)+ \frac {1+\lambda_0}{2}\frac{\log X}{\log x} + O(\log \log \log X).
\end{split}
 \end{align}
   Also, we have
\begin{align}
\label{equ:3.3'}
\begin{split}
 & \log  |L(1/2, \chi)| \leq \Re \left( \sum_{\substack{  p \leq x }} \frac{\chi (p)}{p^{1/2+1/\log x}}
 \frac{\log (x/p)}{\log x} +
 \sum_{\substack{  p \leq \min (x^{1/2}, \log X) }} \frac{\chi (p^2)}{p^{1+2/\log x}}  \frac{\log (x/p^2)}{\log x} \right)
 +\frac{\log X}{\log x} + O(1).
\end{split}
 \end{align}
\end{lemma}

In order to treat the sums over primes in \eqref{equ:3.3} or \eqref{equ:3.3'}, we need a  mean value estimate similar to \cite[Lemma 3]{Sound2009}. For brevity of the statement, we write
\begin{align} \label{sumchiq}
\begin{split}
 \sumstar\limits_{\chi, q} =& \sum_{\substack{(q,3)=1 }}\;  \sumstar_{\substack{\chi \shortmod{q} \\ \chi^3 = \chi_0}} \quad \mbox{or} \quad
 \sum_{\substack{(q,2)=1 }}\;  \sumstar_{\substack{\chi \shortmod{q} \\ \chi^4 = \chi_0}}.
\end{split}
\end{align}
\begin{lemma}
\label{lem:2.5}
With the notations above.  Let $X$ and $y$ be real numbers and l$m$ a positive integer.  For fixed $\varepsilon$ with $0<\varepsilon<1$ and any complex numbers $a(p)$, we have for $j=3,4$,
 \begin{align*}
   \sumstar_{\substack{\chi, q \\ X/2< q \leq X}} & \left|\sum_{\substack{p \leq y}}\frac{a(p)\chi(p)}{p^{1/2}}\right|^{2m} \\
  \ll_{\varepsilon} & X\sum^{\lceil m/j\rceil }_{i=0}m! \binom {m}{j i} \binom {j i}{i}\binom {(j-1) i}{i} a_j
  \Big (\sum_{p \leq y} \frac {|a(p)|^2}{p}\Big )^{m-j i}\Big (\sum_{p \leq y} \frac {|a(p)|^j}{p^{\frac j2}}\Big )^{2i}
  +X^{1/2+\varepsilon}y^{2m+2m\varepsilon} \Big( \sum_{p \leq y} \frac {|a(p)|^2}{p} \Big)^{m},
 \end{align*}
   where
   \[ a_3= \binom {2i}{i}\frac {i!}{36^i} \quad \mbox{and}  \quad a_4= \binom {3i}{i}\frac {(2i)!}{576^i}. \]
\end{lemma}
\begin{proof}
  As the proofs are similar, we again consider only the case involving cubic characters here. Let $W(t)$ be any non-negative smooth function that is supported on $(1/2-\varepsilon_1, 1+\varepsilon_1)$ for some fixed small $0<\varepsilon_1<1/2$ such that $W(t) \gg 1$ for $t \in (1/2, 1)$. We then have
\begin{align*}
  \sum_{\substack{(q,3)=1 \\ X/2< q \leq X}} \ \sumstar_{\substack{\chi \shortmod{q} \\ \chi^3 = \chi_0}} \left|\sum_{\substack{p \leq y}}\frac{a(p)\chi(p)}{p^{1/2}}\right|^{2m} \ll & \sum_{\substack{(q,3)=1 }} \ \sumstar_{\substack{\chi \shortmod{q} \\ \chi^3 = \chi_0}}\left|\sum_{\substack{p \leq y}}\frac{a(p)\chi(p)}{p^{1/2}}\right|^{2m}W \left(\frac {q}{X} \right) \\
= & \sum_{\substack{(q,3)=1 }} \ \sumstar_{\substack{\chi \shortmod{q} \\ \chi^3 = \chi_0}}  \left| \sum_{\substack{p_1, \dots, p_m \leq y}} \frac{a(p_1) \dots a(p_m) \ }{\sqrt{p_1 \dots p_m}}
\chi(p_1\cdots p_m) \right|^2 W \left( \frac {q}{X} \right).
\end{align*}
  We further expand out the square in the last sum above and apply Lemma \ref{PropDirpoly} (by noting that $g(c) \leq 1$ ) to evaluate the resulting sums to see that the last expression above is
\begin{align}
\label{2kbound}
\ll &
X \sum_{\substack{p_1, \dots, p_{2m} \leq y \\ p_1 \dots p_m p^2_{m+1}\dots p^2_{2m} =
\text{cube}}} \frac{|a(p_1) \dots a(p_{2m}) |}{\sqrt{p_1 \dots p_{2m}}} + O\left(X^{1/2+\varepsilon} \sum_{\substack{p_1, \dots, p_{2m} \leq y}}|a(p_1) \dots a(p_{2m})|y^{2m\varepsilon}   \right).
\end{align}
 To estimate the first term above, we note that $p_1 \dots p_m p^2_{m+1}\dots p^2_{2m}= \text{cube}$ precisely when there is a way to partition
the $3m$ primes $\{ p_1, \cdots, p_m, p_{m+1}, \cdots, p_{2m}, p_{m+1}, \cdots, p_{2m}\}$ into groups of three elements so that the corresponding primes in each group are equal.   Such a partition is achieved by first selecting $3i$ indices each from the two sets $\{1, \cdots, m\}$, $\{m+1, \cdots, 2m\}$ and dividing the corresponding primes into groups of three elements and then pairing up those primes whose indices are from the remaining set of $\{1, \cdots, m\}$ with those from the remaining set of $\{m+1, \cdots, 2m\}$. Suppose we divide $3i$ elements for a fixed integer $i$ from each of the sets  $\{1, \cdots, m\}$ and $\{m+1, \cdots, 2m\}$ into small groups of three elements, and pairing up the remaining elements in the set $\{1, \cdots, m\}$ with those in from the set $\{m+1, \cdots, 2m\}$. From this consideration, we see that the number of ways to groups these terms equals
\begin{align*}
 \left( \binom {m}{3i} \frac {(3i)!}{i!6^i} \right)^2 (m-3i)!=\frac {(m!)^2}{(m-3i)!(i!6^i)^2}=m! \binom {m}{3i} \binom {3i}{i}\binom {2i}{i}\frac {i!}{36^i}  .
\end{align*}

  We further note that, in each group of three idential primes, if the indices involved are all from either $\{1, \cdots, m\}$ or  $\{m+1, \cdots, 2m\}$, then these primes will contribute a product of the form $p^3$ in the product of $p_1\cdots p_{2m}$. Otherwise, these primes will contribute a product of the form $p^2$ in the product of $p_1\cdots p_{2m}$.
We thus conclude that we have
\begin{align}
\label{maintermcharbound}
 \sum_{\substack{p_1, \dots, p_{2m} \leq y \\ p_1 \dots p_m p^2_{m+1}\dots p^2_{2m} = \text{cube}}} \frac{|a(p_1) \dots a(p_{2m}) |}{\sqrt{p_1 \dots p_{2m}}} \leq
 \sum^{\lceil m/3\rceil }_{i=0}m! \binom {m}{3i} \binom {3i}{i}\binom {2i}{i}\frac {i!}{36^i} \left(\sum_{p \leq y} \frac {|a(p)|^2}{p}\right)^{m-3i}
 \left(\sum_{p \leq y} \frac {|a(p)|^3}{p^{3/2}}\right)^{2i}.
\end{align}

   On the other hand, the Cauchy-Schwarz inequality gives
\begin{align}
\label{errortermcharbound}
\begin{split}
y^{2m\varepsilon} & \sum_{\substack{p_1, \dots, p_{2m} \leq y}}|a(p_1) \dots a(p_{2m})|  \ll  y^{2m\varepsilon} \left( \sum_{p \leq y} |a(p)| \right)^{2m} \\
  & \ll y^{2m\varepsilon} \left( \sum_{p \leq y} \frac {|a(p)|^2}{p} \right)^{m}\left(\sum_{p \leq y} p \right)^{m} \ll  y^{2m+2m\varepsilon} \left( \sum_{p \leq y} \frac {|a(p)|^2}{p} \right)^{m}.
\end{split}
\end{align}
  Combining \eqref{2kbound}, \eqref{maintermcharbound} and \eqref{errortermcharbound}, we readily deduce the assertion of the lemma. \end{proof}

\section{Proof of Theorem \ref{twistedfirstmoment}}

We only prove \eqref{eq:1} here by modifying the arguments given in the proof of \cite[Theorem 1.1]{B&Y}, as the proof of \eqref{eq:2} follows similarly using arguments in \cite{G&Zhao7}.
We fix $K=\mq(\omega)$ in this section and apply Lemma \ref{lemma:cubicclass} to get that
\begin{equation*} \label{splitting}
\mathcal{M} :=\sum_{(q,3)=1} \ \sumstar_{\substack{\chi \shortmod{q} \\ \chi^3 = \chi_0}} L\left( \frac12, \chi \right) \chi(\ell) \Phi\leg{q}{X}
=
\sumprime_{n \equiv 1 \shortmod{3}} L \left( \frac12, \chi_{n} \right) \chi_n(\ell) \Phi\left(\frac{N(n)}{X}\right),
\end{equation*}
where $\Sigma^{'}$ indicates the sum runs over square-free elements $n$ of $\mathbb{Z}[\omega]$ that have no rational prime divisor. \newline

 We apply the approximate functional equation given in Proposition \ref{prop:AFE} and the notations there by further noting that $\chi(-1)=-1$ in our case to obtain that $\mathcal{M} = \mathcal{M}_1 + \mathcal{M}_2$, where for $AB=X$,
\begin{align*}
 \mathcal{M}_1 =& \sumprime_{n \equiv 1 \shortmod{3}}\ \sum_{m=1}^{\infty} \frac{\chi_n(m\ell)}{\sqrt{m}}  V_{-1}\left(\frac{m}{A} \frac{X}{N(n)} \right)  \Phi \left(\frac{N(n)}{X}\right), \\
 \mathcal{M}_2 =& \sumprime_{n \equiv 1 \shortmod{3}} \epsilon(\chi_n) \sum_{m=1}^{\infty} \frac{\overline{\chi}_n(m\ell)}{\sqrt{m}} V_{-1}\leg{m}{B} \Phi \left(\frac{N(n)}{X}\right).
\end{align*}

Our $M_2$ above is essentially the same as the $M_2$ in \cite[Section 3.3]{B&Y} and thus can be treated essentially the same way.  So following the computations in \cite[Section 3.3]{B&Y}, with minor changes at the appropriate places, we get that
\begin{align} \label{eq:M2estimate}
 \mathcal{M}_2 \ll X^{5/6} B^{1/6}  + X^{2/3} B^{5/6}l^{1/3}.
\end{align}

 To deal with $\mathcal{M}_1$, we use the M\"{o}bius functions as in the proof of Lemma \ref{PropDirpoly} to detect various conditions on $n$ to arrive at
\begin{equation*}
 \mathcal{M}_1 = \sum_{\substack{d \in \mz \\ d \equiv 1 \shortmod{3} }} \mu_{\mz}(d) \sum_{\substack{l \equiv 1 \shortmod{3} \\ (l, d)=1}} \mu_{\omega}(l) \sum_{m=1}^{\infty} \frac{\leg{m\ell}{d l^2}_3}{\sqrt{m}} \mathcal{M}_1(d,l,m, \ell),
\end{equation*}
where
\begin{equation} \label{M1dlm}
 \mathcal{M}_1(d,l,m, \ell) = \sum_{\substack{n \equiv 1 \shortmod{3} \\ (n,d) = 1}} \leg{m\ell}{n}_3 V_{-1}\left(\frac{m}{A} \frac{X}{N(ndl^2)} \right) \Phi\left(\frac{N(ndl^2)}{X}\right).
\end{equation}
  We then apply Mellin inversion to recast \eqref{M1dlm} as
\begin{equation*}
 \mathcal{M}_1(d,l,m, \ell) = \frac{1}{2 \pi i} \int\limits_{(2)} \leg{X}{N(dl^2)}^s L(s, \psi_{m\ell d^3}) \widetilde{f}(s) \dif s,
\end{equation*}
  where
\begin{equation*}
\widetilde{f}(s) = \int\limits_0^{\infty} V_{-1}\left(\frac{m}{Ax} \right) \Phi(x) x^{s-1} \dif x.
\end{equation*}
  We further deduce via the expression for $V_{-1}$ in Proposition \ref{prop:AFE} that
\begin{align}
\label{ftilde}
 \widetilde{f}(1) = \int\limits_0^{\infty} V_{-1}\left(\frac{m}{Ax} \right) \Phi(x) \dif x = \frac{1}{2 \pi i} \int\limits_{(2)} \leg{A}{m}^s \widehat{\Phi}(1+s) \frac{G(s)}{s} \gamma_{-1}(s) \dif s.
\end{align}	

 We estimate $\mathcal{M}_1$ by moving the contour to the line $\Re s = 1/2$ and argue in a similar manner as that in \cite[Section 3.1]{B&Y} to get that the integral on this new line is
\begin{equation} \label{eq:M1'}
 \ll X^{1/2 + \varepsilon} A^{3/4}\ell^{2/3 + \varepsilon}.
\end{equation}
The bound \eqref{eq:M1'} requires the following,
\begin{equation*}
 \sum_{m \leq M} \frac{1}{\sqrt{m}} \left| L \left( \frac{1}{2} + it, \psi_{m\ell d^3} \right) \right| \ll M^{3/4 + \varepsilon}\ell^{2/3+\varepsilon}d^{\varepsilon}(1+|t|)^{2/3 + \varepsilon},
\end{equation*}
which is essentially the same as \cite[(39)]{B&Y} and can be established using similar arguments which utilize the large sieve inequality for cubic characters. \newline

   In the above contour shift, we encounter a pole at $s=1$ when $m\ell$ is a cube. On writing $\ell=\ell_1\ell^2_2\ell^3_3$ with $\ell_1, \ell_2$
   square-free and $(\ell_1, \ell_2)=1$, the contribution from these poles to $\mathcal{M}_1$ equals to
\begin{equation*}
   \mathcal{M}_0 =r_K X \frac 1{\sqrt{\ell^2_1\ell_2}} \sum_{m=1}^{\infty} \frac{\widetilde{f}(1)}{m^{3/2}}
    \prod_{\varpi | 3m\ell d} (1- N(\varpi)^{-1})  \sum_{\substack{d \in \mz, (d,m\ell) =1\\ d \equiv 1 \shortmod{3} }}
     \frac{\mu_{\mz}(d)}{d^2} \sum_{\substack{(l,md\ell)=1 \\ l \equiv 1 \shortmod{3}}} \frac{\mu_{\omega}(l)}{N(l^2)},
\end{equation*}
  where we recall that $r_K$ denotes the residue of $\zeta_K(s)$ at $s = 1$. \newline

Computing the sums over $d$ and $l$ explicitly, we obtain that, for $g(c)$ defined in \eqref{gc} and $c_K$ defined in \eqref{zk},
\begin{align*}
 \mathcal{M}_0 = c_K g(3\ell) X \frac 1{\sqrt{\ell^2_1\ell_2}}  \sum_{m=1}^{\infty} \frac{\widetilde{f}(1)}{m^{3/2}}g \left( \frac {m}{(m,3\ell)} \right).
\end{align*}

   We then apply \eqref{ftilde}, with $m$ there replaced by $m^2\ell^2_1\ell_2$, and get
\begin{align}
\label{M0integral}
 \mathcal{M}_0 = c_K g(3\ell) X \frac 1{\sqrt{\ell^2_1\ell_2}}\frac{1}{2 \pi i} \int\limits_{(2)} \left(\frac {A}{\ell^2_1\ell_2} \right)^s Z_K\left( \frac32 + 3s, \ell \right) \widehat{\Phi}(1+s) \frac{G(s)}{s} \gamma_{-1}(s) ds,
\end{align}
  where $Z_K$ is defined in \eqref{zk}. \newline

 Note that $Z_K(u, \ell)$  is holomorphic and bounded for $\Re (u) \geq 1 + \delta > 1$ and satisfies $Z_K(u, \ell) \ll \ell^{\varepsilon}$ in this region. Thus, we may evaluate the integral in \eqref{M0integral} by moving the contour of integration to $-1/6 + \varepsilon$, crossing a pole at $s=0$ only.  We then deduce that
\begin{equation*}
\mathcal{M}_0 = c_K  g(3\ell) X \frac 1{\sqrt{\ell^2_1\ell_2}}\widehat{\Phi}(1)Z_K\left( \frac32, \ell \right)+O\Big(A^{-1/6 + \varepsilon} (\ell^2_1\ell_2)^{-1/3}\ell^{\varepsilon}X \Big ).
\end{equation*}

  Combining the above with \eqref{eq:M1'}, we see that
\begin{align*}
 \mathcal{M}_1 =c_K  g(3\ell) X \frac 1{\sqrt{\ell^2_1\ell_2}}\widehat{\Phi}(1)Z_K\left( \frac32, \ell \right)+O\left(A^{-1/6 + \varepsilon} (\ell^2_1\ell_2)^{-1/3}\ell^{\varepsilon}X +X^{1/2 + \varepsilon} A^{3/4}\ell^{2/3 + \varepsilon} \right).
\end{align*}
Now \eqref{eq:1} follows from this and \eqref{eq:M2estimate} by setting $A = X^{12/19}$ and $B = X^{7/19}$.  \eqref{eq:2} is proved in a similar way.

\section{Proof of Theorem \ref{thmlowerbound}}
\label{sec 2'}

\subsection{The lower bounds principle}

    We assume that $X$ is a large number throughout the proof and let $\Phi(x)$ be given as in the Introduction. We divide the $q$-range into dyadic blocks so that that to prove Theorem \ref{thmlowerbound}, it suffices to show that
\begin{align*}
   \sumstar_{\chi, q} \left| L\left( \frac{1}{2},\chi \right) \right|^{2k}\Phi \left( \frac qX \right) \gg X(\log X)^{k^2},
\end{align*}
  where $\Phi(x)$ is the same function defined in the Introduction and $\Sigma^{*}_{\chi, q}$ is defined in \eqref{sumchiq}. We point out here
  that throughout our proof, the implicit constants in $\ll$ or $O$ depend on $k$ only and are uniform with respect to $\chi$. \newline

  Let $\{ \ell_j \}_{1 \leq j \leq R}$ be a sequence of even natural
  numbers with $\ell_1= 2\lceil N \log \log X\rceil$ and $\ell_{j+1} = 2 \lceil N \log \ell_j \rceil$ for $j \geq 1$, where $R$ is
  the largest natural number satisfying $\ell_R >10^M$.  Here $N,
  M$ are two large natural numbers depending on $k$ only such that we have $\ell_{j} >
  \ell_{j+1}^2$ for all $1 \leq j \leq R-1$. It follows that we may assume that $M$ is large enough so that
\begin{align}
\label{sumoverell}
  \sum^R_{j=1}\frac 1{\ell_j} \leq \frac 2{\ell_R}, \quad
  (2 r_k+2) \sum^R_{j=1}\frac 1{\ell_j} \leq \frac {4(r_k+1)}{\ell_R} <1,
\end{align}
  where we define $r_k=\lceil k /(2k-1) \rceil+2$ for $k \geq 1$. \newline

Let ${ P}_1$ denote the set of odd primes not exceeding $X^{1/\ell_1^2}$ and
${ P_j}$ denote the set of primes lying in the interval $\left( X^{1/\ell_{j-1}^2}, X^{1/\ell_j^2} \right]$ for $2\le j\le R$.
We define, for $1 \leq j \leq R$ and any real number $\alpha$,
\begin{align*}
& {\mathcal P}_j(\chi) = \sum_{p\in P_j} \frac{1}{\sqrt{p}} \chi(p), \quad  {\mathcal Q}_j(\chi, k) =\Big (\frac{12 k^2 {\mathcal
P}_j(\chi) }{\ell_j}\Big)^{r_k\ell_j}, \quad {\mathcal N}_j(\chi, \alpha) = E_{\ell_j} (\alpha {\mathcal P}_j(\chi)) \quad \mbox{and} \quad \mathcal{N}(\chi, \alpha) = \prod_{j=1}^{R} {\mathcal
N}_j(\chi,\alpha),
\end{align*}
  where
\begin{align}
\label{E}
  E_{\ell}(x) = \sum_{j=0}^{\lceil\ell \rceil} \frac {x^{j}}{j!}
\end{align}
  for any real numbers $x$ and $\ell \geq 0$.
  We also define ${\mathcal Q}_{R+1}(\chi, k)=1$. \newline

  The proof of Theorem \ref{thmlowerbound} depends on the following lower bounds principle of W. Heap and K. Soundararajan and in \cite{H&Sound}.
\begin{lemma}
\label{lem1}
 With the notations above, we have for $k \geq 1/2$,
\begin{align}
\label{basicboundkbig}
\begin{split}
  \sumstar_{\chi, q} L \left( \frac{1}{2},\chi \right) & \mathcal{N}(\chi, k-1) \mathcal{N}(\overline{\chi}, k)\Phi\leg{q}{X}  \\
 \leq & \left( \sumstar_{\chi, q} \left| L \left( \frac{1}{2}, \chi \right) \right|^{2k} \Phi\leg{q}{X}  \right)^{1/(2k)}\left(
 \sumstar_{\chi, q} \prod^R_{j=1} \big ( |{\mathcal N}_j(\chi, k)|^2+ |{\mathcal Q}_j(\chi,k)|^2 \big )\Phi\leg{q}{X}  \right)^{(2k-1)/(2k)}.
\end{split}
\end{align}
  The implied constant in \eqref{basicboundkbig} depends on $k$ only.
\end{lemma}

  We omit the proof of the above lemma since its proof is similar to that of \cite[Lemma 3.1]{Gao2021-4}.   It follows from this lemma that in order to establish Theorem \ref{thmlowerbound}, it suffices to prove the following two propositions.
\begin{proposition}
\label{Prop4} With the notations above, we have for $k>0$,
\begin{align*}
\sumstar_{\chi, q}L(\tfrac{1}{2},\chi) \mathcal{N}(\overline{\chi}, k) \mathcal{N}(\chi, k-1)\Phi\leg{q}{X}  \gg X(\log X)^{ k^2
} .
\end{align*}
\end{proposition}

\begin{proposition}
\label{Prop6} With the notations above, we have for $k>0$,
\begin{align}
\label{N2estmation}
\sumstar_{\chi, q} \prod^R_{j=1}\big ( |{\mathcal N}_j(\chi, k)|^2+ |{\mathcal Q}_j(\chi,k)|^2 \big ) \Phi\leg{q}{X} \ll X(\log X)^{ k^2
} .
\end{align}
\end{proposition}

In the next two sections, we give proofs of the above two propositions. As the proofs are similar for cubic and quartic characters, we only consider
the case for cubic characters.

\subsection{Proof of Proposition \ref{Prop4}}
\label{sec 4}

Let $w(n)$ be the multiplicative function with $w(p^{\alpha}) = \alpha!$ for prime powers $p^{\alpha}$ and write $\Omega(n)$ for the number of distinct prime powers dividing $n$.  Moroever, let $b_j(n)$, $1 \leq j \leq R$ be the function such that $b_j(n)\in \{ 0 ,1 \}$ and that $b_j(n)=1$ if and only if when $n$ is composed of at most $\ell_j$ primes, all from the interval $P_j$. \newline

    These notations allow us to write for any real number $\alpha$,
\begin{equation}
\label{5.1}
{\mathcal N}_j(\chi, \alpha) = \sum_{n_j} \frac{1}{\sqrt{n_j}} \frac{\alpha^{\Omega(n_j)}}{w(n_j)}  b_j(n_j) \chi(n_j), \quad 1\le j\le R.
\end{equation}

    As $b_j(n_j)=0$ unless $n_j \leq
    (X^{1/\ell_j^2})^{\ell_j}=X^{1/\ell_j}$, each ${\mathcal N}_j(\chi, \alpha)$ is a short Dirichlet polynomial. As a consequence, both ${\mathcal N}(\chi, k)$ and ${\mathcal N}(\chi, k-1)$ are short Dirichlet
    polynomials with lengths at most $X^{1/\ell_1+ \ldots +1/\ell_R} < X^{2/10^{M}}$ by \eqref{sumoverell}. Moreover,
     for each $\chi$ modulo $q$,
\begin{align*}
 & {\mathcal N}(\overline{\chi}, k){\mathcal N}(\chi, k-1) \ll X^{2(1/\ell_1+ \ldots +1/\ell_R)} < X^{4/10^{M}}.
\end{align*}

  Write for simplicity that
\begin{align*}
 {\mathcal N}(\chi, k-1)= \sum_{a  \leq X^{2/10^{M}}} \frac{x_a}{\sqrt{a}} \chi(a) \quad \mbox{and} \quad \mathcal{N}(\overline{\chi}, k) = \sum_{b  \leq
 X^{2/10^{M}}} \frac{y_b}{\sqrt{b}}\overline{\chi}(b).
\end{align*}

  We now apply Theorem \ref{twistedfirstmoment} to obtain that
\begin{align}
\label{L1lowerbound}
\begin{split}
 \sum_{(q,3)=1} \ \sumstar_{\substack{\chi \shortmod{q} \\ \chi^3 = \chi_0}} L \left( \frac{1}{2}, \chi \right) \mathcal{N}(\overline{\chi}, k) \mathcal{N}(\chi, k-1)\Phi\leg{q}{X}
= & \sum_{a} \sum_{b} \frac {x_a y_b}{\sqrt{ab}}\sum_{(q,3)=1} \ \sumstar_{\substack{\chi \shortmod{q} \\ \chi^3 = \chi_0}} L \left( \frac{1}{2}, \chi \right)\chi(ab^2) \Phi\leg{q}{X}\\
\gg & \sum_{a} \sum_{b} \frac {x_a y_b}{\sqrt{ab}} g(3ab^2) X \frac 1{\sqrt{(ab^2)^2_1(ab^2)_2}}Z_{\mq(\omega)} \left( \frac32, ab^2 \right).
\end{split}
\end{align}
We remind the reader here that the subscripts in $(ab^2)_1$ and $(ab^2)_2$ are used in the same sense as the decomposition of $l$ in the statement of Theorem~\ref{twistedfirstmoment}.  Namely, the positive integer $ab^2$ is uniquely written as the product of positive integers $(ab^2)_1$, $(ab^2)_2^2$ and $(ab^2)_3^3$ with $(ab^2)_1$, $(ab^2)_2$ both square-free and relatively prime.  We note that the contribution of the error term arising from \eqref{eq:1} in the above process is negligible since that $a, b \leq  X^{2/10^{M}}$ and $x_a, y_b \ll 1$. \newline

  We further write $ g(3ab^2) Z_{\mq(\omega)}(3/2, ab^2)$ as a constant multiple of $h(ab^2)$,
  where $h$ is a multiplicative function satisfying 
  \begin{equation} \label{hest}
  h(p^i)=1+O(1/p)
  \end{equation}
  for all primes $p$ and integers $i \geq 1$.
  This then implies that the last expression in \eqref{L1lowerbound} is
\begin{align}
\label{sumab}
\begin{split}
\gg & \sum_{a} \sum_{b} \frac {x_a y_b}{\sqrt{ab}} h(ab^2) \frac X{\sqrt{(ab)^2_1(ab^2)_2}} \\
=& X \prod^R_{j=1}\left( \sum_{n_j, n'_j} \frac{1}{\sqrt{n_jn'_j}} \frac{1}{\sqrt{(n_j{n'_j}^2)^2_1(n_j{n'_j}^2)_2}}h(n_j{n'_j}^2)\frac{(k-1)^{\Omega(n_j)}}{w(n_j)}  b_j(n_j) \frac{k^{\Omega(n'_j)}}{w(n'_j)}
 b_j(n'_j) \right).
\end{split}
\end{align}

 For a fixed $j$ with $1 \leq j \leq R$ in \eqref{sumab}, we consider the sum above over $n_j, n'_j$ by noting that the factors $b_j(n_j), b_j(n'_j)$ restricts $n_j, n'_j$ to have all prime factors in $P_j$ such that $\Omega(n_j), \Omega(n'_j) \leq \ell_j$.
  If we remove these restrictions on $\Omega(n_j)$ and $\Omega(n'_j)$, then the sum involves with multiplicative functions so that one evaluates
  it to be
\begin{align*}
\begin{split}
&  \prod_{\substack{p_i \in P_j }}\left( \sum_{\alpha_i, \alpha'_i \geq 0}  \frac{1}{\sqrt{p^{\alpha_i+\alpha'_i}_{i}}}\frac{1}{\sqrt{(p^{\alpha_i+2\alpha'_i}_{i})^2_1(p^{\alpha_i+2\alpha'_i}_{i})_2}} h \left( p^{\alpha_i+2\alpha'_i}_{i} \right) \frac{(k-1)^{\alpha_i}}{\alpha_i!}  \frac{k^{\alpha'_i}}{\alpha'_i!}
  \right).
\end{split}
\end{align*}

Inspecting the inner sum above, we see that only the terms in which $\alpha_i=\alpha'_i=1$ and $\alpha_i=0$, $\alpha'_i=1$ yield a constant multiple of $1/p_i$; the other terms in which either $\alpha_i$ or $\alpha'_i \geq 1$ leads to terms of size $\ll p^{-3/2}$, mindful of \eqref{hest}.  We then recast the above expression as
\begin{align}
\label{6.02}
\begin{split}
&  \prod_{\substack{p\in P_j }}\left(1+ \frac {k(k-1)+k}p+O \left( \frac 1{p^{3/2}} \right) \right).
\end{split}
\end{align}

   To estimate the error introduced in the above process, we apply Rankin's trick by noticing that $2^{\Omega(n_j)-\ell_j}\ge 1$ if $\Omega(n_j) > \ell_j$. Thus we see that the error is
\begin{align}
\label{errorbound}
\begin{split}
\le & \sum_{n_j, n'_j} \frac{1}{\sqrt{n_jn'_j}} \frac{1}{\sqrt{(n_j{n'_j}^2)^2_1(n_j{n'_j}^2)_2}}h \left( n_j{n'_j}^2 \right)\frac{k^{\Omega(n_j)}2^{\Omega(n_j)-\ell_j}}{w(n_j)} \frac{(1-k)^{\Omega(n'_j)}}{w(n'_j)}
  \\
\le & 2^{-\ell_j} \prod_{\substack{p\in P_j }} \left( 1+ \frac {2k(1-k)}p+O \left( \frac 1{p^2} \right) \right) \leq 2^{-\ell_j/2}\prod_{\substack{p\in P_j }}\left(1+ \frac {k^2}p+O \left( \frac 1{p^2} \right) \right),
\end{split}
\end{align}
  where we obtain the last estimation above by observing that it follows from Lemma \ref{RS} that when $N$ is large enough, we have
\begin{align}
\label{boundsforsumoverp}
   \sum_{p \in P_j}\frac 1{p} \leq \frac{\ell_j}{N} .
\end{align}

  We then deduce from \eqref{6.02}, \eqref{errorbound} and Lemma \ref{RS} that we have
\begin{align*}
\sum_{(q,3)=1} \ \sumstar_{\substack{\chi \shortmod{q} \\ \chi^3 = \chi_0}} & L(1/2,\chi) \mathcal{N}(\overline{\chi}, k) \mathcal{N}(\chi, k-1)\Phi\leg{q}{X} \\
\gg & X \prod^R_{j=1}\Big (1+O( 2^{-\ell_j/2})\Big )\prod^R_{j=1}\prod_{\substack{p\in P_j }}\left(1+ \frac {k(k-1)+k}p+O\left( \frac 1{p^2} \right) \right) \gg
X(\log X)^{k^2}.
\end{align*}
 This completes the proof of the proposition.

\subsection{Proof of Proposition \ref{Prop6}}

   We apply Lemma \ref{PropDirpoly} to evaluate the left-hand of \eqref{N2estmation} and
   we may ignore the contribution from the error term in \eqref{cubicandquarticcharsum} in this process in view of \eqref{sumoverell}.
   Using the expression for ${\mathcal N}_j(\chi, \alpha)$ in
   \eqref{5.1}, we see that
\begin{align}
\label{maintermbound}
\begin{split}
 \sum_{(q,3)=1}\;  \sumstar_{\substack{\chi \shortmod{q} \\ \chi^3 = \chi_0}} \prod^R_{j=1} & \Big ( |{\mathcal N}_j(\chi, k)|^2+ |{\mathcal Q}_j(\chi,k)|^2 \Big ) \Phi\leg{q}{X} \\
\ll & X \prod^R_{j=1} \left(  \sum_{\substack{n_j, n'_j \\ n_j{n'_j}^2=\text{cube}}}
\frac{k^{\Omega(n_j)+\Omega(n'_j)}}{\sqrt{n_jn'_j} w(n_j) w(n'_j)}  b_j(n_j)b_j(n'_j)  \right. \\
& \hspace*{3cm} \left. + \Big( \frac{12k^2} {\ell_j}\Big)^{2r_k\ell_j}((r_k\ell_j)!)^2 \sum_{ \substack{ \Omega(n_j)=\Omega(n'_j) = r_k\ell_j \\ p|n_jn'_j \implies  p\in P_j\\n_j{n'_j}^2=\text{cube} }} \frac{1 }{\sqrt{n_jn'_j} w(n_j) w(n'_j)} \right).
\end{split}
\end{align}

  Arguing as in the proof of Proposition \ref{Prop4}, we get that
\begin{align} \label{sqinN}
\begin{split}
   \sum_{\substack{n_j, n'_j \\ n_j{n'_j}^2=\text{cube}}} \frac{k^{\Omega(n_j)+\Omega(n'_j)}}{\sqrt{n_jn'_j}  w(n_j) w(n'_j)}  b_j(n_j)b_j(n'_j) =\Big(1+ O\big(2^{-\ell_j/2} \big ) \Big)
  \exp \left( \sum_{p \in P_j} \frac {k^2}{p}+ O\left( \sum_{p \in P_j} \frac {1}{p^{3/2}} \right) \right).
\end{split}
\end{align}
   Note also that,
\begin{align*}
\begin{split}
\sum_{ \substack{ \Omega(n_j)=\Omega(n'_j) = r_k\ell_j \\ p|n_jn'_j \implies  p\in P_j\\n_j{n'_j}^2=\text{cube} }} \frac{1 }{\sqrt{n_jn'_j}  w(n_j) w(n'_j)}   \leq &  \sum_{ \substack{  p|n_jn'_j \implies  p\in P_j \\n_j{n'_j}^2=\text{cube} }} \frac{(12 k^2 r_k)^{\Omega(n_j)+\Omega(n'_j)-2r_k\ell_j}  }{\sqrt{n_jn'_j}  w(n_j) w(n'_j)} \\
\leq & (12 k^2 r_k)^{-2r_k\ell_j} \prod_{p \in P_j} \left( 1+  \frac {(12 k^2 r_k)^2}{p}+ O \left( \frac {1}{p^{3/2}} \right) \right).
\end{split}
\end{align*}

 To treat the term $((r_k\ell_j)!)^2$ in \eqref{maintermbound} and for later purposes, we note the following estimations.
\begin{align}
\label{Stirling}
 \binom {n}{m} \leq \left( \frac {en}{m} \right)^m, \quad \left( \frac ne \right)^n \leq n! \leq n \left( \frac ne \right)^n.
\end{align}
  We apply the above and \eqref{boundsforsumoverp} to see that by taking $M, N$ large enough,
\begin{align} \label{Qest}
\begin{split}
 \Big( \frac{12k^2} {\ell_j}\Big)^{2r_k\ell_j}((r_k\ell_j)!)^2 \sum_{ \substack{ \Omega(n_j)=\Omega(n'_j) = r_k\ell_j \\ p|n_jn'_j \implies  p\in P_j\\n_j{n'_j}^2=\text{cube} }} \frac{1 }{\sqrt{n_jn'_j}  w(n_j) w(n'_j)} \ll & (r_k\ell_j)^2e^{-2r_k\ell_j}\prod_{p \in P_j} \left( 1+  \frac {(12 k^2 r_k)^2}{p}+ O\left( \frac {1}{p^{3/2}} \right) \right) \\
\ll & e^{-\ell_j} \exp \left( \sum_{p \in P_j} \frac {k^2}{p}+ O \left( \sum_{p \in P_j} \frac {1}{p^2} \right) \right).
\end{split}
\end{align}
 The assertion of the proposition now follows by using \eqref{sqinN} and \eqref{Qest} in \eqref{maintermbound} and then applying Lemma \ref{RS}.

\section{Proof of Theorem \ref{thmupperbound}}
\label{sec:upper bd}

\subsection{A first treatment}

  In the course of proving Theorem \ref{thmupperbound}, we need to first establish some weaker upper bounds for moments of the related families of $L$-functions in this section.
  Let $X$ be a large number and set $\mathcal{N}_3(V, X)$ (or $\mathcal{N}_4(V, X)$) to be the number of primitive cubic (or quartic)
  Dirichlet characters $\chi \bmod q$ whose square remain primitive such that $X/2 < q\leq X$ and $\log|L(1/2, \chi)|\geq V$.
  Our estimations require the following upper bounds for $\mathcal{N}_i(V, X), i=3, 4$.

\begin{prop}
\label{propNbound}
Assume the truth of GRH for $\zeta(s)$ and for $L(s,\chi)$ for all primitive cubic and quartic Dirichlet characters $\chi$. Let $i=3, 4$  and $k>0$ be a fixed real number.
 If $10\sqrt{\operatorname{log}\operatorname{log} X}\leq V \leq 10^{4+4k} \log \log X$,  then
 \begin{align*}
  \mathcal{N}_i(V,X)\ll X(\log \log X)\operatorname{exp}\left(-\frac{V^2}{\log \log X}
  \left(1-\frac{2\cdot 10^{6+4k}}{\operatorname{log}\operatorname{log}\operatorname{log}X}\right)\right).
 \end{align*}
  If $10^{4+4k} \log \log X <V \leq 6\log X/\log \log X$, we have
 \begin{align*}
  \mathcal{N}_i (V,X)\ll XV^2 \operatorname{exp}\left(-(2+4k) V \right).
 \end{align*}
\end{prop}
\begin{proof}
  As the proofs are similar, we only prove the case of $i=3$ here. We apply \eqref{equ:3.3} by setting $x=X^{A/V}$ there with
 $A =\log\log\log X/10^{6+4k}$. We further let $z=x^{1/\log \log X}$.   Write $M_1$ for the real part of the sum in \eqref{equ:3.3} truncated to $p \leq z$ and $M_2$ the complementary sum over $z < p \leq x$.  It then follows that
\[
 \log  |L(1/2, \chi)|
    \leq M_1 + M_2 + \frac {1+\lambda_0}{2}V + O(\log \log \log X).
\]
Hence if  $\log  |L(1/2, \chi)| \geq V $, then we have either
\[ M_2 \geq \frac{V}{8A} \quad \mbox{or} \quad M_1 \geq V_1 := V \left(1-\frac{7}{8A} \right). \]

Now, we set
\[ \operatorname{meas}(X;M_1) = \# \{\text{primitive cubic Dirichlet characters with conductors not exceeding  $X$}: M_1 \geq V_1 \} \]
and
\[ \operatorname{meas}(X;M_2) = \# \left\{ \text{primitive cubic Dirichlet characters with conductors not exceeding  $X$}: M_2 \geq \frac{V}{8A}  \right\}. \]

We take $m = \lceil V/8A \rceil $ in Lemma \ref{lem:2.5}, where $\lceil x \rceil = \min \{ n \in \intz : n \geq x\}$, getting
\begin{align} \label{M1bound}
\begin{split}
(V/8A)^{2m} & \operatorname{meas}(X;M_2) \leq \sum_{\substack{(q,3)=1 \\ X/2<q \leq X}} \ \sumstar_{\substack{\chi \shortmod{q} \\ \chi^3 = \chi_0}}|M_2|^{2m} \\
 \ll & X\sum^{\lceil m/3\rceil }_{i=0}m! \binom {m}{3i} \binom {3i}{i}\binom {2i}{i}\frac {i!}{36^i} \Big (\sum_{z<p \leq x} \frac {1}{p}\Big )^{m-3i}\Big (\sum_{z<p \leq x} \frac {1}{p^{3/2}}\Big )^{2i}+X^{1/2+\varepsilon}x^{2m+2m\varepsilon} \Big( \sum_{z<p \leq x} \frac {1}{p} \Big)^{m}.
\end{split}
\end{align}

  We note that
\begin{align}
\label{p32}
 \sum_{p } \frac {1}{p^{3/2}} \leq \sum_{n \geq 2} \frac {1}{n^{3/2}} \leq 2.
\end{align}

   Moreover, applying \eqref{Stirling}, we see that for $i \geq 1$,
\begin{align*}
  m! \binom {m}{3i} \binom {3i}{i}\binom {2i}{i}\frac {i!}{36^i} \leq   m \Big (\frac {m}{e} \Big )^m \frac {m^{3i}}{i^{2i}}.
\end{align*}
  As the function $x \mapsto m^{3x}/x^{2x}$ is increasing for $x \leq m/3$ if $m \geq e$, we deduce from this and check directly for the case $i=0$ that we have for all $i \geq 0$ and $m \geq e$,
\begin{align}
\label{binombound}
  m! \binom {m}{3i} \binom {3i}{i}\binom {2i}{i}\frac {i!}{36^i} \leq  m^{4m/3+1}.
\end{align}

  It follows from this, \eqref{M1bound}, \eqref{p32} and Lemma~\ref{RS},
\begin{align*}
 & \Big( \frac {V}{8A} \Big)^{2m} \operatorname{meas}(X;M_2)
 \ll  X m^2 (2m)^{4m/3}    \left(\sum_{\substack{z<p  \leq x}}
\frac{1}{p}\right)^m  \ll X m^2 (2m)^{4m/3}   \Big (\log \log \log X+O(1) \Big )^m.
\end{align*}

   We then deduce from the above that
\begin{equation}
\label{equ:bd-S-2}
\operatorname{meas}(X;M_2) \ll X\operatorname{exp}\left(-\frac{V}{20A}\log V\right).
\end{equation}

Next, we estimate $\operatorname{meas}(X;M_1)$.  Taking $m=\lceil V^2_1/\log \log X \rceil $ if $V \leq 10^{4+4k}(\log \log X)$ and $ m = \lceil V\rceil $ otherwise, we have, for $X$ large,
\begin{align*}
m \leq \frac{(1/2-0.1)\log X}{ \log z}.
\end{align*}
Now Lemma \ref{lem:2.5}, \eqref{p32} and \eqref{binombound} give
\begin{align}
\label{equ:3.4}
\begin{split}
V^{2m}_1\operatorname{meas}(X;M_1) \leq  &  \sum_{\substack{(q,3)=1 \\ X/2<q \leq X}} \ \sumstar_{\substack{\chi \shortmod{q} \\ \chi^3 = \chi_0}} |M_1|^{2m} \ll X\sum^{\lceil m/3\rceil }_{i=0}m! \binom {m}{3i} \binom {3i}{i}\binom {2i}{i}\frac {i!}{36^i} \Big (\sum_{p \leq z} \frac {1}{p}\Big )^{m-3i}\Big (\sum_{p \leq z} \frac {1}{p^{3/2}}\Big )^{2i} \\
\ll & X m \Big (\frac {m}{e} \Big )^m \Big (\log \log X \Big )^{m}\Big (1+ \sum^{\lceil m/3\rceil }_{i=1}\Big (\frac {m 2^{2/3}}{i^{2/3}}\Big )^{3i}\Big (\log \log X \Big )^{-3i}\Big ).
\end{split}
\end{align}
  Note if $V \leq 10^{4+4k}(\log \log X)$, then
\begin{align*}
\begin{split}
 \Big (\frac {m 2^{2/3}}{i^{2/3}}\Big )^{3i}\Big (\log \log X \Big )^{-3i} \leq \Big (\frac {10^{9+8k}}{i^{2/3}}\Big )^{3i}.
\end{split}
\end{align*}
When summed over over $i \geq 1$, the expression on the right-hand side of the above gives a convergent series.  Hence we conclude from this and \eqref{equ:3.4} that if $V \leq 10^{4+4k}(\log \log X)$,
\begin{align}
\label{M1bound1}
\operatorname{meas}(X ; M_1)
\ll X m \left(\frac{m \log \log X}{eV_1^2} \right)^m \ll  X(\log \log X)\operatorname{exp}\left(-\frac{V_1^2}{\log \log X} \right).
\end{align}

   For $V \geq 10^{4+4k}(\log \log X)$, we consider the function
\begin{align*}
\begin{split}
  s(y)= \Big (\frac {m 2^{2/3}}{y^{2/3}}\Big )^{3y}\Big (\log \log X \Big )^{-3y} .
\end{split}
\end{align*}
Differentiating, we get that $s(y)$ is increasing if
\begin{equation} \label{incrcond}
3 \log \left( \frac{m 2^{2/3}}{ e^{2/3} \log \log X} \right) \geq 2 \log y .
\end{equation}
As $1 \leq y \leq m/3$ under our consideration, \eqref{incrcond} certainly holds if $m \geq (e/6)^2(\log \log X)^3$ so that $s(y)$ achieves its maximal value in the interval $[1,m/3]$ at $y=m/3$ with the maximum value
\begin{align*}
\begin{split}
  \Big ( m^{1/3} 6^{2/3}\Big )^{m}\Big (\log \log X \Big )^{-m}.
\end{split}
\end{align*}
   This implies that when $V \geq 10^{4+4k}(\log \log X)$ and $\lceil V\rceil  \geq (e/6)^2(\log \log X)^3$, we have
\begin{align}
\label{M1bound2}
\begin{split}
\operatorname{meas}(X ; M_1)
\ll& X m \Big (\frac {m\log \log X}{e V_1^2} \Big )^m \Big (1+m \Big ( m^{1/3} 6^{2/3}\Big )^{m}\Big (\log \log X \Big )^{-m} \Big )
\ll X V \left(\frac{1}{10^{2+4k}}\right)^{V}+X V^2\Big (\frac {V^{4/3}}{V^2_1} \Big )^{V} \\
\ll&  XV^2 \operatorname{exp}\left(-(2+4k) V \right).
\end{split}
\end{align}

  On the other hand, the function $s(y)$ is maximized when $m \leq (e/6)^2(\log \log X)^3$ at $y = (m/\log \log X)^{3/2}2/e$ with the maximal value
\[   \exp \left( \frac{4}{e} \left( \frac{m}{\log \log X} \right)^{3/2} \right). \]

  This implies that when $V \geq 10^{4+4k}(\log \log X)$ and $\lceil V\rceil  \leq (e/6)^2(\log \log X)^3$, we have
\begin{align}
\label{M1bound3}
\begin{split}
\operatorname{meas}(X ; M_1)
\ll& X m \Big (\frac {m\log \log X}{e V_1^2} \Big )^m \left( 1+m \exp \left( \frac{4}{e} \left( \frac{m}{\log \log X} \right)^{3/2} \right) \right).
\end{split}
\end{align}
  Note that if $\lceil V\rceil  \leq (e/6)^2(\log \log X)^3$, then
\begin{align*}
\begin{split}
( m/\log \log X)^{3/2} \leq  ( V/\log \log X)^{3/2} \leq V \left( \frac{ \lceil V \rceil}{(\log \log X)^3} \right)^{1/2} \leq \frac {e}{5} V.
\end{split}
\end{align*}
  It follows from this and \eqref{M1bound3} that if $V \geq 10^{4+4k}(\log \log X)$, then
\begin{align*}
\begin{split}
\operatorname{meas}(X ; M_1)
\ll& X m^2 \Big (\frac {m\log \log X}{e V_1^2} \Big )^m e^{4V/5} \ll X V^2 \exp \left(-(2+4k) V \right).
\end{split}
\end{align*}

 One checks that for our choice of $A$, we have
\begin{align*}
\exp \left(-\frac{V}{20A}\log V\right) \leq
\left\{
 \begin{array}
  [c]{ll}
\exp \left(-\frac{V^2}{\log \log X} \right), \quad V \leq 10^{4+4k} \log \log X, \\ \\
 \exp\left(-(2+4k) V \right), \quad V > 10^{4+4k} \log \log X.
\end{array}
\right .
\end{align*}

  The assertion of the proposition now follows from \eqref{equ:bd-S-2}, \eqref{M1bound1} and \eqref{M1bound2}. \end{proof}

  Now, Proposition \ref{propNbound} allows us to establish the following weaker upper bounds for moments of the $L$-functions under our consideration.

\begin{prop}
\label{prop: upperbound}
Assume RH for $\zeta(s)$ and GRH for $L(s, \chi)$ for all primitive cubic and quartic Dirichlet characters.  Let $k$ be a positive integer and $\varepsilon>0$ be given.  We have, for large $X$,
\begin{align*}
    \sum_{\substack{(q,3)=1 \\ X/2< q \leq X}}\;  \sumstar_{\substack{\chi \shortmod{q} \\ \chi^3 = \chi_0}} \left| L \left( \frac{1}{2}, \chi \right) \right|^{2k}  \ll_k  X(\log X)^{k^2+\varepsilon} \quad \mbox{and} \quad
    \sum_{\substack{(q,2)=1 \\ X/2< q \leq X}}\;  \sumstar_{\substack{\chi \shortmod{q} \\ \chi^4 = \chi_0}} \left| L \left( \frac{1}{2}, \chi \right) \right|^{2k}  \ll_k  X(\log X)^{k^2+\varepsilon}.
\end{align*}
\end{prop}
\begin{proof}
  As the proofs are similar, we again consider only the case for cubic characters here. We write $\mathcal{N}(V,X)$ for $\mathcal{N}_3(V,X)$
 and note that
\begin{align}
\label{momentint}
  \sum_{\substack{(q,3)=1 \\ X/2 < q \leq X}} \ \sumstar_{\substack{\chi \shortmod{q} \\ \chi^3 = \chi_0}} \left| L \left( \frac{1}{2}, \chi \right) \right|^{2k}  =& -\int\limits_{-\infty}^{+\infty}\exp (2kV) \dif \mathcal{N}(V,X)
  = 2k\int\limits_{-\infty}^{+\infty}\operatorname{exp}(2kV)\mathcal{N}(V,X) \dif V,
\end{align}
after integration by parts.  As $N(V, X) \ll X$, we see that
\begin{align*}
 2k\int\limits_{-\infty}^{10\sqrt{\log \log X}}\operatorname{exp}(2kV)\mathcal{N}(V,X) \dif V \ll X\int\limits_{-\infty}^{10\sqrt{\log \log X}}\operatorname{exp}(2kV) \dif V \ll X(\log X)^{k^2}.
\end{align*}
Thus it now remains to consider the $V$-range with $10\sqrt{\log \log X} \leq V$.  By taking $x = \log X$ in (\ref{equ:3.3}) and bounding the sum  over $p$ in (\ref{equ:3.3}) trivially, we see that $\mathcal{N} (V,X) =0$  for $V > 6\log X/\log \log X$. Thus, we can also assume that $V \leq 6\log X/\log \log X$. \newline

 We then apply Proposition \ref{propNbound} to see that for $10\sqrt{\log \log X} \leq V \leq 6\log X/\log \log X$,
\begin{align}
 \label{equ:rough-01}
  \mathcal{N}(V, X)\ll
\begin{cases}
  X(\operatorname{log}X)^{o(1)}\operatorname{exp}\left(-\frac{V^2}{\log \log X}\right), \quad 10\sqrt{\log \log X} \leq V \leq 10^{4+4k} \log \log X, \\ \\
   X(\operatorname{log}X)^{o(1)}\operatorname{exp}(-(2+4k)V), \quad  V > 10^{4+4k} \log \log X.
\end{cases}
\end{align}
  Applying the bounds given in \eqref{equ:rough-01} to evaluate the integral in \eqref{momentint} now leads to the assertion of Proposition~\ref{prop: upperbound}. \end{proof}

\subsection{Completion of the proof}
\label{sect 3.4}

  Upon dividing $q$ into dyadic blocks, we may assume that $X/2 < q \leq X$.
  Once again we only consider the case of cubic Dirichlet $L$-functions here.
   We start by taking exponentials on both sides of the upper bound for $\log  |L(\tfrac{1}{2}, \chi)|$ given in \eqref{equ:3.3'}.  This gives
\begin{align}
\label{basicest}
\begin{split}
 & \left| L \left( \frac{1}{2}, \chi \right) \right|^{2k}  \ll
 \exp \left(2k \Re \left( \sum_{\substack{  p \leq x }} \frac{\chi (p)}{p^{1/2+1/\log x}}
 \frac{\log (x/p)}{\log x} +
 \sum_{\substack{  p \leq \min (x^{1/2}, \log X) }} \frac{\chi (p^2)}{p^{1+2/\log x}}  \frac{\log (x/p^2)}{\log x}
 +\frac{\log X}{\log x}\right) \right ).
\end{split}
 \end{align}
   We would like to estimate the sums on the right-hand side of \eqref{basicest} using an approach similar to that in the proof of Theorem \ref{thmlowerbound}, by dividing the sums into different ranges of $p$.  But the situation is slightly different here, due to the presence of the parameter $x$.  For this reason, we follow the approach by A. J. Harper in \cite{Harper} and define for a large number $T$,
$$ \alpha_{0} = \frac{\log 2}{\log X}, \;\;\;\;\; \alpha_{i} = \frac{20^{i-1}}{(\log\log X)^{2}} \;\;\; \forall \; i \geq 1, \quad
\mathcal{J} = \mathcal{J}_{k,X} = 1 + \max\{i : \alpha_{i} \leq 10^{-T} \} . $$

   We set
\[ {\mathcal M}_{i,j}(\chi) = \sum_{X^{\alpha_{i-1}} < p \leq X^{\alpha_{i}}}  \frac{\chi (p)}{p^{1/2+1/(\log X^{\alpha_{j}})}} \frac{\log (X^{\alpha_{j}}/p)}{\log X^{\alpha_{j}}}, \quad 1\leq i \leq j \leq \mathcal{J} , \]
and
\[ P_{m}(\chi )= \sum_{2^{m} < p \leq 2^{m+1}} \frac{\chi (p)}{p^{1+2/(\log X^{\alpha_{j}})}} \frac{\log (X^{\alpha_{j}}/p^2)}{\log X^{\alpha_{j}}}, \quad  0 \leq m \leq \frac{\log\log X}{\log 2}. \]

Let $\mathcal{C}(X)$ be the set of primitive cubic Dirichlet characters of conductor $q$ with $X/2 < q \leq X$.  We also define for $1\leq j \leq \mathcal{J}$,
\begin{align*}
 \mathcal{S}(j) =& \left\{ \chi \in \mathcal{C}(X) : | \Re {\mathcal M}_{i,l}(\chi)| \leq \alpha_{i}^{-3/4} \; \; \mbox{for all}  \; 1 \leq i \leq j, \; \mbox{and} \; i \leq l \leq \mathcal{J}, \right. \\
 & \left. \;\;\;\;\; \text{but }  |\Re {\mathcal M}_{j+1,l}(\chi)| > \alpha_{j+1}^{-3/4} \; \text{ for some } j+1 \leq l \leq \mathcal{J} \right\} . \\
 \mathcal{S}(\mathcal{J}) =& \left\{ \chi \in \mathcal{C}(X) : |\Re{\mathcal M}_{i, \mathcal{J}}(\chi)| \leq \alpha_{i}^{-3/4} \; \mbox{for all}  \; 1 \leq i \leq \mathcal{J} \right\}, \; \mbox{and} \\
\mathcal{P}(m) =&  \left\{ \chi \in \mathcal{C}(X) : |\Re P_{m}(\chi)| > 2^{-m/10} , \; \text{but} \; |\Re P_{n}(\chi)| \leq 2^{-n/10} \; \mbox{for all} \; m+1 \leq n \leq \frac{\log\log X}{\log 2} \right\}.
\end{align*}

    We shall set $x=X^{\alpha_j}$ for $j \geq 1$ in \eqref{basicest} in what follows.  So we may assume that the second summation on the right side of \eqref{basicest} is over $p \leq \log X$.  Then we have
$|\Re P_{n}(\chi)| \leq 2^{-n/10}$ for all $n$ if $\chi \not \in \mathcal{P}(m)$ for any $m$, which implies that
\[ \Re \sum_{\substack{  p \leq \log X }} \frac{\chi (p)}{p^{1+2/\log x}}  \frac{\log (x/p^2)}{\log x}  = O(1). \]
 As the treatment for case $\chi \not \in \mathcal{P}(m)$ for any $m$ is easier compared to the other cases, we may assume that $\chi \in \mathcal{P}(m)$ for some $m$.   We further note that
$$ \mathcal{P}(m)= \bigcup_{m=0}^{\log \log X/2}\bigcup_{j=0}^{ \mathcal{J}} \Big (\mathcal{S}(j)\bigcap \mathcal{P}(m) \Big ), $$
so that it suffices to show that
\begin{align}
\label{sumovermj}
  \sum_{m=0}^{\log \log X/2}\sum_{j=0}^{\mathcal{J}}\sum_{\chi \in \mathcal{S}(j)\bigcap \mathcal{P}(m)} |L(1/2, \chi)|^{2k}
   \ll X (\log X)^{k^{2}} .
\end{align}

     Let $W(t)$ be defined as in the proof
of Lemma~\ref{lem:2.5}, we have that
\begin{align*}
\text{meas}(\mathcal{P}(m)) \leq \sum_{\substack{(q,3)=1}} \ \sumstar_{\substack{\chi \shortmod{q} \\ \chi^3 = \chi_0}}
\Big (2^{m/10} |P_m(\chi)| \Big )^{2\lceil 2^{m/2}\rceil } W \left( \frac qX \right).
\end{align*}
We use the same approach as in the proof of Lemma \ref{lem:2.5} and the estimate \eqref{binombound} to bound the right-hand side of the above expression.  This gives, for $m \geq 10$,
\begin{align}
\label{Pmest}
\begin{split}
  \text{meas}(\mathcal{P}(m)) \ll & X \sum^{\lceil \lceil 2^{m/2}\rceil /3\rceil }_{i=0}
  \Big (\lceil 2^{m/2}\rceil /3\Big )^{4\lceil 2^{m/2}\rceil /3+1} \Big (\sum_{2^m < p } \frac {1}{p^2}\Big )^{\lceil 2^{m/2}\rceil -3i}
  \Big (\sum_{2^m < p } \frac {1}{p^3}\Big )^{2i} \\
   \ll & X 2^m (2^{2m/3})^{\lceil 2^{m/2}\rceil }\Big (\sum_{2^m < p } \frac {1}{p^2}\Big )^{\lceil 2^{m/2}\rceil } \ll X 2^m (2^{-m/3})^{2^{m/2}} \ll X  2^{-2^{m/2}}.
\end{split}
 \end{align}
  We then apply the Cauchy-Schwarz inequality and Proposition \ref{prop: upperbound} to see that when $2^{m} \geq (\log\log X)^{3}$,
\begin{align*}
\sum_{\chi \in  \mathcal{P}(m)} \left| L \left( \frac{1}{2}, \chi \right) \right|^{2k} \leq & \left( \text{meas}(\mathcal{P}(m)) \cdot
\sum_{\substack{(q,3)=1 \\ X/2<q \leq X}}  \sumstar_{\substack{\chi \shortmod{q} \\ \chi^3 = \chi_0}}|L(1/2, \chi)|^{4k} \right)^{1/2}
 \\
 \ll & \left( X \exp\left( -(\log 2)(\log\log X)^{3/2} \right) X (\log X)^{(2k)^{2}+1} \right)^{1/2} \ll X(\log X)^{k^2}.
\end{align*}

   The above implies  that we may also assume that $0 \leq m \leq (3/\log 2)\log\log\log X$.
We further note that, for $W(t)$ defined as in the proof of Lemma~\ref{lem:2.5},
\begin{align}
\label{S0est}
\begin{split}
\text{meas}(\mathcal{S}(0)) \ll & \sum_{\substack{(q,3)=1 }} \ \sumstar_{\substack{\chi \shortmod{q} \\ \chi^3 = \chi_0}}
\sum^{\mathcal{J}}_{l=1}
\Big ( \alpha^{3/4}_{1}{|\mathcal
M}_{1, l}(\chi)| \Big)^{2\lceil 1/(10\alpha_{1})\rceil }W \left( \frac qX \right) \\
=&
\sum^{\mathcal{J}}_{l=1}\sum_{\substack{(q,3)=1 }} \ \sumstar_{\substack{\chi \shortmod{q} \\ \chi^3 = \chi_0}}\Big ( \alpha^{3/4}_{1}{|\mathcal
M}_{1, l}(\chi)| \Big)^{2\lceil 1/(10\alpha_{1})\rceil }W \left( \frac qX \right) .
\end{split}
\end{align}
   Note that we have
\begin{align}
\label{sump1}
 \mathcal{J} \leq \log\log\log X , \; \alpha_{1} = \frac{1}{(\log\log X)^{2}} , \; \mbox{and} \; \sum_{p \leq X^{1/(\log\log X)^{2}}} \frac{1}{p} \leq \log\log X ,
\end{align}
  where the last estimation follows from Lemma \ref{RS}. We apply these estimates to evaluate the last sums in \eqref{S0est} above in a manner similar to the approach in the proof of Theorem \ref{thmlowerbound}.  This yields
\begin{align*}
\text{meas}(\mathcal{S}(0)) \ll &
\mathcal{J}X e^{-1/\alpha_{1}}\ll X e^{-(\log\log X)^{2}/10}  .
\end{align*}
  We then deduce via the Cauchy-Schwarz inequality and Proposition \ref{prop: upperbound}  that
\begin{align*}
\sum_{\chi \in  \mathcal{S}(0)} \left| L \left( \frac{1}{2}, \chi \right) \right|^{2k}  \leq &   \left( \text{meas}(\mathcal{S}(0)) \cdot
\sum_{\substack{(q,3)=1 \\ X/2<q \leq X}}  \sumstar_{\substack{\chi \shortmod{q} \\ \chi^3 = \chi_0}}|L(\half, \chi)|^{4k} \right)^{1/2}
 \\
 \ll & \left( X \exp\left( -(\log\log X)^{2}/10 \right) X (\log X)^{(2k)^{2}+1} \right)^{1/2} \ll X(\log X)^{k^2}.
\end{align*}

  Thus we may further assume that $j \geq 1$. Note that when $\chi \in \mathcal{S}(j)$, we set $x=X^{\alpha_j}$ in \eqref{basicest} to arrive at
\begin{align*}
\begin{split}
 & \left| L \left( \frac{1}{2}, \chi \right) \right|^{2k} \ll \exp \left(\frac {2k}{\alpha_j} \right) \exp \Big (
 2k\Re\sum^j_{i=0}{\mathcal M}_{i,j}(\chi)+2k \Re\sum^{\log \log X/2}_{m=0}P_m(\chi) \Big ).
\end{split}
 \end{align*}

   When restricting the sum of $|L(1/2, \chi)|^{2k}$ over $\mathcal{S}(j)\bigcap \mathcal{P}(m)$, our treatments below require us to separate the sums over $p \leq 2^{m+1}$ on the right-hand side of hte above expression from those over $p>2^{m+1}$. For this, we note that if $\chi \in \mathcal{P}(m)$, then
\begin{align}
\label{sump}
\begin{split}
 & \Re \sum_{  p \leq 2^{m+1}}  \frac{\chi (p)}{p^{1/2+1/(\log X^{\alpha_{j}})}}  \frac{\log (X^{\alpha_{j}}/p)}{\log X^{\alpha_{j}}}+
  \Re \sum_{p \leq \log X} \frac{\chi (p)}{p^{1+2/(\log X^{\alpha_{j}})}}  \frac{\log (X^{\alpha_{j}}/p^2)}{\log X^{\alpha_{j}}}
   \\
 \leq &\Re \sum_{  p \leq 2^{m+1}}  \frac{\chi (p)}{p^{1/2+1/(\log X^{\alpha_{j}})}}  \frac{\log (X^{\alpha_{j}}/p)}{\log X^{\alpha_{j}}}+
 \Re \sum_{p \leq 2^{m+1}} \frac{\chi (p)}{p^{1+2/(\log X^{\alpha_{j}})}}  \frac{\log (X^{\alpha_{j}}/p^2)}{\log X^{\alpha_{j}}}+O(1) \leq 2^{m/2+3}+O(1).
\end{split}
 \end{align}

Then it follows from the above that
\begin{align}
\label{LboundinSP}
\begin{split}
  \sum_{\chi \in \mathcal{S}(j)\bigcap \mathcal{P}(m)} & \left| L \left( \frac{1}{2}, \chi \right) \right|^{2k} \\
   \ll & e^{k2^{m/2+4}} \sum_{\chi \in \mathcal{S}(j)\bigcap \mathcal{P}(m)}
 \exp \left(\frac {2k}{\alpha_j} \right)\exp \Big ( 2k \Re \sum_{ 2^{m+1}< p \leq X^{\alpha_j} }
 \frac{\chi (p)}{p^{\tfrac{1}{2}+1/(\log X^{\alpha_{j}})}}  \frac{\log (X^{\alpha_{j}}/p)}{\log X^{\alpha_{j}}}\Big )
  \\
\ll &  e^{k2^{m/2+4}} \exp \left(\frac {2k}{\alpha_j} \right)\sum_{\chi \in \mathcal{S}(j)} \Big (2^{m/10}|P_m(\chi)| \Big )^{2\lceil 2^{m/2}\rceil }
\exp \Big ( 2k \Re{\mathcal M}'_{1,j}(\chi)+2k\Re \sum^j_{i=2}{\mathcal M}_{i,j}(\chi)\Big ),
\end{split}
 \end{align}
   where we define
\begin{align*}
\begin{split}
 {\mathcal M}'_{1,j}(\chi)= \sum_{ 2^{m+1}< p \leq X^{\alpha_1} }
 \frac{\chi (p)}{p^{1/2+1/\log X^{\alpha_{j}}}}  \frac{\log (X^{\alpha_{j}}/p)}{\log X^{\alpha_{j}}}.
\end{split}
 \end{align*}

    We note that if $0 \leq m \leq (3/\log 2)\log\log\log X$ and $X$ large enough, then
\begin{align*}
\begin{split}
  \sum_{ p< 2^{m+1}  }
 \frac{\chi (p)}{p^{1/2+1/\log X^{\alpha_{j}}}}  \frac{\log (X^{\alpha_{j}}/p)}{\log X^{\alpha_{j}}} \leq \sum_{ p< 2^{m+1}  }
 \frac{1}{\sqrt{p}} \leq \frac {100 \cdot 2^{m/2}}{m+1} \leq 100(\log \log X)^{3/2}(\log \log \log X)^{-1},
\end{split}
 \end{align*}
   where the last estimation above follows from partial summation and \eqref{PIT}. \newline

  It follows from this that if $\chi \in \mathcal{S}(0)$ and $X$ large enough,
\begin{align}
\label{M'est}
\begin{split}
 {\mathcal M}'_{1,j}(\chi)\leq 100(\log \log X)^{3/2}(\log \log \log X)^{-1}+{\mathcal M}_{1,j}(\chi) \leq 1.01\alpha^{-3/4}_1=1.01(\log \log X)^{3/2}.
\end{split}
 \end{align}

  As we also have ${\mathcal M}_{i, j} \leq  \alpha^{-3/4}_i$ if $\chi \in \mathcal{S}(0)$, we can apply \cite[Lemma 5.2]{Kirila}.  This gives
\begin{align*}
\begin{split}
\exp \Big ( 2k\Re {\mathcal M}'_{1,j}(\chi)+2k \Re\sum^j_{i=2}{\mathcal M}_{i,j}(\chi)\Big ) \ll \Big | E_{e^2k\alpha^{-3/4}_1}(k{\mathcal M}'_{1,j}(\chi)) \Big |^2
\prod^j_{i=2}\Big | E_{e^2k\alpha^{-3/4}_i}(k{\mathcal M}_{i,j}(\chi)) \Big |^2,
\end{split}
 \end{align*}
  where $E_{e^2k\alpha^{-3/4}_i}$ is defined in \eqref{E}. \newline

   We then deduce from the description on $\mathcal{S}(j)$ that when $j \geq 1$,
\begin{align}
\label{L2k}
\begin{split}
  \sum_{\chi \in \mathcal{S}(j)\bigcap \mathcal{P}(m)} & \left| L \left( \frac{1}{2}, \chi \right) \right|^{2k}  \\
 \ll & e^{k2^{m/2+4}}\exp \left( \frac {2k}{\alpha_j} \right)
 \sum^{R}_{l=j+1} \sum_{\substack{(q,3)=1 \\ X/2<q \leq X}} \ \sumstar_{\substack{\chi \shortmod{q} \\ \chi^3 = \chi_0}}\Big (2^{m/10}|P_m(\chi)| \Big )
 ^{2\lceil 2^{m/2}\rceil } \\
& \hspace*{2cm} \times \exp \Big ( 2k \Re {\mathcal M}'_{1,j}(\chi)+2k \Re \sum^j_{i=2}{\mathcal M}_{i,j}(\chi)\Big )\Big ( \alpha^{3/4}_{j+1}{\mathcal
M}_{j+1, l}(\chi)\Big)^{2\lceil 1/(10\alpha_{j+1})\rceil } \\
\ll &  e^{k2^{m/2+4}}\exp \left(\frac {2k}{\alpha_j} \right) \sum^{R}_{l=j+1}
\sum_{\substack{(q,3)=1 \\ X/2<q \leq X}} \ \sumstar_{\substack{\chi \shortmod{q} \\ \chi^3 = \chi_0}}
\Big (2^{m/10}|P_m(\chi)| \Big )^{2\lceil 2^{m/2}\rceil } \\
& \hspace*{2cm} \times \Big | E_{e^2k\alpha^{-3/4}_1}(k{\mathcal M}'_{1,j}(\chi)) \Big |^2
\prod^j_{i=2}\Big | E_{e^2k\alpha^{-3/4}_i}(k{\mathcal M}_{i,j}(\chi)) \Big |^2\Big ( \alpha^{3/4}_{j+1}{\mathcal
M}_{j+1, l}(\chi)\Big)^{2\lceil 1/(10\alpha_{j+1})\rceil } .
\end{split}
 \end{align}

 Note that we have for $1 \leq j \leq \mathcal{I}-1$,
\begin{align}
\label{sumpj}
\mathcal{I}-j \leq \frac{\log(1/\alpha_{j})}{\log 20}  \; \mbox{and} \; \sum_{X^{\alpha_{j}} < p \leq X^{\alpha_{j+1}}} \frac{1}{p}
 = \log \alpha_{j+1} - \log \alpha_{j} + o(1) = \log 20 + o(1) \leq 10 .
\end{align}
Therefore we argue in a manner similar to the proof of Theorem \ref{thmlowerbound} and make use of \eqref{Pmest}.  Upon taking $T$ large enough,
\begin{align*}
\begin{split}
  \sum^{ \mathcal{I}}_{l=j+1}
\sum_{\substack{(q,3)=1 \\ X/2<q \leq X}} & \ \sumstar_{\substack{\chi \shortmod{q} \\ \chi^3 = \chi_0}}
\Big (2^{m/10}|P_m(\chi)| \Big )^{2\lceil 2^{m/2}\rceil } \\
& \hspace*{1cm} \times \Big | E_{e^2k\alpha^{-3/4}_1}(k{\mathcal M}'_{1,j}(\chi)) \Big |^2
\prod^j_{i=2}\Big | E_{e^2k\alpha^{-3/4}_i}(k{\mathcal M}_{i,j}(\chi)) \Big |^2\Big ( \alpha^{3/4}_{j+1}{\mathcal
M}_{j+1, l}(\chi)\Big)^{2\lceil 1/(10\alpha_{j+1})\rceil } \\
\ll & X(\mathcal{I}-j)e^{-44k/\alpha_{j+1}} 2^m (2^{-2m/15})^{\lceil 2^{m/2}\rceil } \prod_{p \leq X^{\alpha_j}}\left( 1+\frac {k^2}{p}+O \left( \frac 1{p^2} \right) \right) \\
\ll  &  e^{-42k/\alpha_{j+1}} 2^m (2^{-2m/15})^{\lceil 2^{m/2}\rceil }X (\log X)^{k^{2}}.
\end{split}
 \end{align*}

   We then conclude from the above and \eqref{LboundinSP} that (by noting that $20/\alpha_{j+1}=1/\alpha_j$)
\begin{align*}
\begin{split}
 & \sum_{\chi \in \mathcal{S}(j)\bigcap \mathcal{P}(m)} |L(1/2, \chi)|^{2k}
\ll   e^{-k/(10\alpha_{j})}2^m  e^{k2^{m/2+4}} (2^{-2m/15})^{\lceil 2^{m/2}\rceil }X (\log X)^{k^{2}}.
\end{split}
 \end{align*}

   As the sum of the right side expression over $m$ and $j$ converges, we see that the above implies \eqref{sumovermj}
and this completes the proof of Theorem \ref{thmupperbound}.

\section{Proof of Theorem \ref{thmlowerbound1}}
\label{Sec: Pf of Thmlowerbounds}

  As the proof is similar to that of Theorem \ref{thmupperbound}, we shall omit some details to avoid repetition.  Once again we only consider the cubic case in what follows.  We keep the notations in the proofs of Theorems \ref{thmlowerbound} and \ref{thmupperbound} and define
\begin{align*}
   {\mathcal P}'_i(\chi)=&  \sum_{X^{\alpha_{i-1}} < p \leq X^{\alpha_{i}}}  \frac{\chi (p)}{\sqrt{p}}, \quad \mbox{and} \quad {\mathcal Q}'_i(\chi, k)=\Big( \frac{12 |{\mathcal P}'_i(\chi)|}{\lceil e^2k\alpha^{-3/4}_i \rceil}\Big)^{r'_k\lceil e^2k\alpha^{-3/4}_i \rceil},
\end{align*}
   where $r'_k=\lceil 1+1/k \rceil+1$.
  We also define for any real number $\alpha$ and any $1\leq i  \leq \mathcal{J}$,
\begin{align*}
 {\mathcal M}_i(\chi, \alpha) = E_{e^2k\alpha^{-3/4}_i} \Big (\alpha {\mathcal P}'_i(\chi) \Big ), \quad  {\mathcal M}(\chi, \alpha)=  \prod^{\mathcal{J}}_{i=1} {\mathcal M}_i(\chi, \alpha).
\end{align*}
  Note that each ${\mathcal M}_i(\chi, \alpha)$ is a short Dirichlet polynomial of length at most $X^{\alpha_{i}\cdot e^2k\alpha^{-3/4}_i}=X^{e^2k\alpha^{1/4}_i}$. By taking $X$ large enough, we have that
\begin{align*}
 \sum^{\mathcal{J}}_{i=1} e^2k\alpha^{1/4}_i \leq 2e^2k10^{-T/4}.
\end{align*}
   It follows that ${\mathcal M}(\chi, \alpha)$ is also a short Dirichlet polynomial of length at most $X^{2e^2k10^{-T/4}}$. \newline

   Note additionally that we have by \eqref{sump1} and \eqref{sumpj},
\begin{align*}
   \sum_{X^{\alpha_{i-1}} < p \leq X^{\alpha_{i}}}\frac 1{p} \leq \frac{100}{10^{3T/4}}\alpha^{-3/4}_i, \quad 1\leq i  \leq \mathcal{J}.
\end{align*}

  Instead of using products involving $\mathcal{N}$ in the lower bounds principle as given in Lemma \ref{lem1}, we apply the same principle to products with $\mathcal{M}$.  We do this for the case $0 \leq k < 1/2$ as follows.
\begin{lemma}
\label{lem1'}
 With the notations above, we have for $0 \leq k < 1/2$,
\begin{align}
\label{basiclowerbound}
\begin{split}
\sumstar_{\chi, q} L \left( \frac{1}{2}, \chi \right) & \mathcal{M}(\chi, k-1) \mathcal{M}(\overline{\chi}, k)\Phi\leg{q}{X} \\
 \ll & \Big ( \sumstar_{\chi, q} \left| L \left( \frac{1}{2}, \chi \right) \right|^{2k}   \Phi\leg{q}{X} \Big )^{1/2}\Big ( \sumstar_{\chi, q} \left| L \left( \frac{1}{2}, \chi \right) \right|^2 |\mathcal{M}(\chi, k-1)|^2 \Phi\leg{q}{X}  \Big)^{(1-k)/2} \\
 & \hspace*{2cm} \times \Big ( \sumstar_{\chi, q} \prod^{\mathcal{J}}_{i=1} \big ( |{\mathcal M}_i(\chi, k)|^2+ |{\mathcal Q}'_i(\chi,k)|^2 \big )\Phi\leg{q}{X}
 \Big)^{k/2}.
\end{split}
\end{align}
  The implied constant in \eqref{basiclowerbound} depends on $k$ only.
\end{lemma}
\begin{proof}
  The proof is similar to that of \cite[Lemma 3.1]{Gao2021-4}.  We first use H\"older's inequality to bound the left side of \eqref{basiclowerbound} as
\begin{align}
\label{basicbound0}
\begin{split}
 \leq & \Big ( \sumstar_{\chi, q} \Big| L \left( \frac{1}{2}, \chi \right) \Big|^{2k} \Big )^{1/2}\Big ( \sumstar_{\chi, q} \Big| L \Big( \frac{1}{2}, \chi \Big) \mathcal{M}(\chi, k-1)|^2   \Big)^{(1-k)/2}\Big ( \sumstar_{\chi, q} |\mathcal{M}(\chi,
 k)|^{2/k}|\mathcal{M}(\chi, k-1)|^{2}  \Big)^{k/2}.
\end{split}
\end{align}
  As in the proof of \cite[Lemma 3.1]{Gao2021-4}, we note for $|z| \le aK/10$ with $0<a \leq 1$,
\begin{align}
\label{Ebound}
\Big| \sum_{r=0}^K \frac{z^r}{r!} - e^z \Big| \le \frac{|z|^{K}}{K!} \le \Big(\frac{a e}{10}\Big)^{K},
\end{align}
  We apply \eqref{Ebound} with $z=k{\mathcal P}'_i(\chi)$, $K=e^2k\alpha^{-3/4}_i$ and $a=k$, getting that if $|{\mathcal P}'_i(\chi)| \le \lceil e^2k\alpha^{-3/4}_i \rceil/10$, then
\begin{align*}
{\mathcal M}_i(\chi, k)=& \exp ( k{\mathcal P}'_i(\chi) )( 1+  O \left( \exp ( k |{\mathcal P}'_i(\chi)| ) \left( \frac{k e}{10} \right)^{e^2k\alpha^{-3/4}_i} \right) \\
= & \exp ( k {\mathcal P}'_i(\chi)  ) \left( 1+  O\left( ke^{-e^2k\alpha^{-3/4}_i} \right) \right).
\end{align*}
  Similarly, we have
\begin{align}
\label{MiP}
{\mathcal M}_i(\chi, k-1)= & \exp \left( (k-1) {\mathcal P}'_i(\chi) \right) \left( 1+  O\left( e^{-e^2k\alpha^{-3/4}_i} \right)  \right).
\end{align}

   The above estimates yield that if $|{\mathcal P}'_i(\chi)| \le \lceil e^2k\alpha^{-3/4}_i \rceil/10$, then
\begin{align}
\label{est1}
\begin{split}
|{\mathcal M}_i(\chi, k)^{\frac {1}{k}} {\mathcal M}_i(\chi, k-1)|^{2}
=& \exp ( 2k \Re  {\mathcal P}'_i(\chi)  ) \left( 1+ O\big( e^{-e^2k\alpha^{-3/4}_i} \big) \right) \\
=& |{\mathcal M}_j(\chi, k)|^2 \left( 1+ O\big(e^{-e^2k\alpha^{-3/4}_i} \big)
\right) .
\end{split}
\end{align}

  On the other hand, if $| {\mathcal P}'_i(\chi)  \ge \lceil e^2k\alpha^{-3/4}_i \rceil/10$,
\begin{equation} \label{MestPlarge}
\begin{split}
|{\mathcal M}_i(\chi, k)| \le \sum_{r=0}^{\lceil e^2k\alpha^{-3/4}_i \rceil} \frac{|{\mathcal P}'_i(\chi)|^r}{r!} & \le
|{\mathcal P}'_i(\chi)|^{\lceil e^2k\alpha^{-3/4}_i \rceil} \sum_{r=0}^{\lceil e^2k\alpha^{-3/4}_i \rceil} \Big( \frac{10}{\lceil e^2k\alpha^{-3/4}_i \rceil}\Big)^{\lceil e^2k\alpha^{-3/4}_i \rceil-r} \frac{1}{r!}  \\
&   \le \Big( \frac{12 |{\mathcal
P}'_i(\chi)|}{\lceil e^2k\alpha^{-3/4}_i \rceil}\Big)^{\lceil e^2k\alpha^{-3/4}_i \rceil} .
\end{split}
\end{equation}
  Observe that the same bound above also holds for $|{\mathcal M}_i(\chi, k-1)|$. It follows from these estimations that if $|{\mathcal
  P}'_i(\chi)| \ge \lceil e^2k\alpha^{-3/4}_i \rceil/10$, we have
\begin{align*}
|{\mathcal M}_i(\chi, k)^{\frac {1}{k}} {\mathcal M}_i(\chi, k-1)|^{2}
& \leq \Big( \frac{12 |{\mathcal P}'_i(\chi)|}{\lceil e^2k\alpha^{-3/4}_i \rceil}\Big)^{2(1+1/k)\lceil e^2k\alpha^{-3/4}_i \rceil} \leq  |{\mathcal Q}'_i(\chi, k)|^2.
\end{align*}
 Applying the above and \eqref{est1} to \eqref{basicbound0}, the assertion of the lemma follows.
\end{proof}

  We deduce from Lemma \ref{lem1'} that in order to prove Theorem \ref{thmlowerbound1}, it suffices to show that
\begin{align}
\label{LMM} \sumstar_{\chi, q}L \left( \frac{1}{2},\chi \right) \mathcal{M}(\overline{\chi}, k) \mathcal{M}(\chi, k-1)\Phi\leg{q}{X}  \gg & X(\log X)^{ k^2
},  \\
\label{MQ} \sumstar_{\chi, q} \prod^{\mathcal{J}}_{i=1} \big ( |{\mathcal M}_i(\chi, k)|^2+ |{\mathcal Q}'_i(\chi,k)|^2 \big ) \Phi\leg{q}{X} \ll & X(\log X)^{ k^2
}, \\
\label{LM} \sumstar_{\chi, q} \left| L \left( \frac{1}{2},\chi \right) \right|^2 |\mathcal{M}(\chi, k-1)|^2 \Phi\leg{q}{X} \ll & X(\log X)^{ k^2
}.
\end{align}

  The estimates in \eqref{LMM} and \eqref{MQ} can be established similar to Proposition \ref{Prop4} and Proposition \ref{Prop6}, respectively. To prove \eqref{LM}, we argue in a manner similar to the treatments in Section \ref{sect 3.4} and conclude that it suffices to show that
\begin{align}
\label{sumovermj1}
  \sum_{m=0}^{(3/\log 2)\log\log\log X}\sum_{j=1}^{\mathcal{J}}\sum_{\chi \in \mathcal{S}(j)\bigcap \mathcal{P}(m)} |L(\tfrac{1}{2},\chi)|^2 |\mathcal{M}(\chi, k-1)|^2
   \ll X (\log X)^{k^{2}} .
\end{align}

   Similar to \eqref{L2k}, we get that
\begin{align}
\label{LMbound}
\begin{split}
 \sum_{\chi \in \mathcal{S}(j)\bigcap \mathcal{P}(m)} & |L(\tfrac{1}{2},\chi)|^2 |\mathcal{M}(\chi, k-1)|^2 \\
\ll & e^{k2^{m/2+4}}\exp \left(\frac {2k}{\alpha_j} \right) \sum^{R}_{l=j+1}
\sum_{\substack{(q,3)=1 \\ X/2<q \leq X}} \ \sumstar_{\substack{\chi \shortmod{q} \\ \chi^3 = \chi_0}}
\Big (2^{m/10}|P_m(\chi)| \Big )^{2\lceil 2^{m/2}\rceil } \\
& \hspace*{1.5cm} \times \Big | E_{e^2k\alpha^{-3/4}_1}(k{\mathcal M}'_{1,j}(\chi)) \Big |^2\Big | E_{e^2k\alpha^{-3/4}_1}((k-1){\mathcal P}'_{1}(\chi)) \Big |^2 \\
& \hspace*{1.5cm} \times \prod^j_{i=2}\Big | E_{e^2k\alpha^{-3/4}_i}(k{\mathcal M}_{i,j}(\chi)) \Big |^2 | E_{e^2k\alpha^{-3/4}_i}((k-1){\mathcal P}'_{i}(\chi)) \Big |^2 \Big ( \alpha^{3/4}_{j+1}{\mathcal
M}_{j+1, l}(\chi)\Big)^{2\lceil 1/(10\alpha_{j+1})\rceil } .
\end{split}
 \end{align}

 As in the proof of Theorem \ref{thmupperbound},  when treating the right-hand side of \eqref{LMbound}, we want to separate the terms with $p \leq 2^{m+1}$ and those with $p>2^{m+1}$.  For this, we deduce, similar to \eqref{MiP}, that, if $|{\mathcal P}'_i(\chi)| \le \lceil e^2k\alpha^{-3/4}_i \rceil/10$,
\begin{align}
\label{EPrimboundPsmall}
E_{e^2k\alpha^{-3/4}_1}((k-1){\mathcal P}'_{1}(\chi)) \ll \exp ( (k-1) {\mathcal P}'_i(\chi)  ).
\end{align}
Similar to \eqref{sump} and \eqref{M'est},
\begin{align} \label{psumestsmallandlarge}
 \sum_{p \leq 2^{m+1}}\frac{\chi (p)}{\sqrt{p}} \ll 2^{m/2}, \quad \mbox{and} \quad \sum_{2^{m+1}<p \leq X^{\alpha_1}}\frac{\chi (p)}{\sqrt{p}} \leq \alpha^{-3/4}_1.
\end{align}
  It follows from \cite[Lemma 5.2]{Kirila} that
\begin{align}
\label{Expbound}
 \exp ( (k-1) {\mathcal P}'_i(\chi)  ) \ll e^{(1-k)2^{m/2}} E_{e^2k\alpha^{-3/4}_1}\Big((k-1)\sum_{2^{m+1}<p \leq X^{\alpha_1}}\frac{\chi (p)}{\sqrt{p}}\Big ).
\end{align}

   Combining \eqref{EPrimboundPsmall} and \eqref{Expbound}, we conclude that if $|{\mathcal
  P}'_i(\chi)| \le \lceil e^2k\alpha^{-3/4}_i \rceil/10$, then
\begin{align}
\label{EPrimboundPsmall1}
E_{e^2k\alpha^{-3/4}_1}((k-1){\mathcal P}'_{1}(\chi)) \ll e^{(1-k)2^{m/2}} E_{e^2k\alpha^{-3/4}_1}\Big((k-1)\sum_{2^{m+1}<p \leq X^{\alpha_1}}\frac{\chi (p)}{\sqrt{p}}\Big ).
\end{align}

  When $|{\mathcal
  P}'_i(\chi)| \ge \lceil e^2k\alpha^{-3/4}_i \rceil/10$, very much similar to \eqref{MestPlarge}, we arrive at
\begin{align}
\label{EPprimebound}
\begin{split}
E_{e^2k\alpha^{-3/4}_1}((k-1){\mathcal P}'_{1}(\chi)) &\le  \Big( \frac{12 |{\mathcal
P}'_1(\chi)|}{\lceil e^2k\alpha^{-3/4}_1 \rceil}\Big)^{\lceil e^2k\alpha^{-3/4}_1 \rceil} .
\end{split}
\end{align}

Now \eqref{psumestsmallandlarge} gives that if $m \leq (3/\log 2)\log\log\log X$ with $X$ large enough, then
\begin{align*}
\begin{split}
  \left| \sum_{p \leq 2^{m+1}}\frac{\chi (p)}{\sqrt{p}} \right| \leq 100(\log \log X)^{3/2}(\log \log \log X)^{-1} \leq \frac {\lceil e^2k\alpha^{-3/4}_i \rceil }{20} \leq \frac 12|{\mathcal P}'_1(\chi)|.
\end{split}
\end{align*}
  It follows that
\begin{align*}
\begin{split}
  \left| \sum_{2^{m+1}<p \leq X^{\alpha_1}}\frac{\chi (p)}{\sqrt{p}} \right| \geq  \left| {\mathcal P}'_1(\chi) \right| - \left| \sum_{p \leq 2^{m+1}}\frac{\chi (p)}{\sqrt{p}} \right| \geq \frac 12 \left| {\mathcal P}'_1(\chi) \right|.
\end{split}
\end{align*}

   We deduce from this and \eqref{EPprimebound} that when $|{\mathcal
  P}'_i(\chi)| \ge \lceil e^2k\alpha^{-3/4}_i \rceil/10$,
\begin{align*}
\begin{split}
E_{e^2k\alpha^{-3/4}_1}((k-1){\mathcal P}'_{1}(\chi)) &\le  \left( \frac{24  |\sum_{2^{m+1}<p \leq X^{\alpha_1}}\frac{\chi (p)}{\sqrt{p}}|}{\lceil e^2k\alpha^{-3/4}_1 \rceil}\right)^{\lceil e^2k\alpha^{-3/4}_1 \rceil} .
\end{split}
\end{align*}

  We conclude from the above and \eqref{EPrimboundPsmall1} that
\begin{align*}
\Big|E_{e^2k\alpha^{-3/4}_1} & \left( (k-1){\mathcal P}'_{1}(\chi) \right)\Big |^2 \\
& \ll e^{(1-k)2^{m/2+1}}  \left| E_{e^2k\alpha^{-3/4}_1}\Big((k-1)\sum_{2^{m+1}<p \leq X^{\alpha_1}}\frac{\chi (p)}{\sqrt{p}}\Big ) \right|^2  +\left|\left( \frac{24  |\sum_{2^{m+1}<p \leq X^{\alpha_1}}\frac{\chi (p)}{\sqrt{p}}|}{\lceil e^2k\alpha^{-3/4}_1 \rceil}\right)^{\lceil e^2k\alpha^{-3/4}_1 \rceil}\right|^2.
\end{align*}

   Substituting the above into \eqref{LMbound}, we see that
\begin{align}
\label{LM2}
\begin{split}
 \sum_{\chi \in \mathcal{S}(j)\bigcap \mathcal{P}(m)} & \left| L \left( \frac{1}{2} ,\chi \right) \right|^2 |\mathcal{M}(\chi, k-1)|^2 \\
\ll &  e^{k2^{m/2+4}}\exp \left(\frac {2k}{\alpha_j} \right)  \sum^{R}_{l=j+1}
\sum_{\substack{(q,3)=1 \\ X/2<q \leq X}} \ \sumstar_{\substack{\chi \shortmod{q} \\ \chi^3 = \chi_0}}
\Big (2^{m/10}|P_m(\chi)| \Big )^{2\lceil 2^{m/2}\rceil }  \Big | E_{e^2k\alpha^{-3/4}_1}(k{\mathcal M}'_{1,j}(\chi)) \Big |^2\\
& \hspace*{1cm} \times\Big (  e^{(1-k)2^{m/2+1}} \Big|E_{e^2k\alpha^{-3/4}_1}\Big((k-1)\sum_{2^{m+1}<p \leq X^{\alpha_1}}\frac{\chi (p)}{\sqrt{p}}\Big )\Big|^2+\Big|\Big( \frac{24  |\sum_{2^{m+1}<p \leq X^{\alpha_1}}\frac{\chi (p)}{\sqrt{p}}|}{\lceil e^2k\alpha^{-3/4}_1 \rceil}\Big)^{\lceil e^2k\alpha^{-3/4}_1 \rceil}\Big|^2 \Big ) \\
& \hspace*{1cm} \times \prod^j_{i=2}\Big | E_{e^2k\alpha^{-3/4}_i}(k{\mathcal M}_{i,j}(\chi)) \Big |^2 | E_{e^2k\alpha^{-3/4}_i}((k-1){\mathcal P}'_{i}(\chi)) \Big |^2 \Big ( \alpha^{3/4}_{j+1}{\mathcal
M}_{j+1, l}(\chi)\Big)^{2\lceil 1/(10\alpha_{j+1})\rceil } \\
& \hspace*{1cm} \times \prod^{\mathcal{J}}_{i=j+1}| E_{e^2k\alpha^{-3/4}_i}((k-1){\mathcal P}'_{i}(\chi)) \Big |^2  .
\end{split}
 \end{align}

  Now, proceeding as in the proofs of Propositions \ref{Prop4} and \ref{Prop6} and noting that, similar to the estimation given in \eqref{MQ}, the factor $\prod^{\mathcal{J}}_{i=j+2}| E_{e^2k\alpha^{-3/4}_i}((k-1){\mathcal P}'_{i}(\chi)) \Big |^2$ of the right hand side expression above gives rise to a contribution that is
\begin{align*}
\begin{split}
\ll & \prod_{X^{\alpha_{j+2}} \leq p \leq X^{\alpha_{\mathcal{J}}}}\left( 1+\frac {(k-1)^2}{p}+O \left( \frac 1{p^2} \right) \right) \ll \exp \Big ( \sum_{X^{\alpha_{j+2}} \leq p \leq X^{\alpha_{\mathcal{J}}}} \frac {(k-1)^2}{p} \Big ) \ll \exp \Big ( -(k-1)^2 \log \alpha_{j+2} \Big ),
\end{split}
 \end{align*}
   where the last estimation above follows from Lemma \ref{RS}.  We may choose $X$ large enough so that this is
\begin{align*}
\begin{split}
\ll & \exp \left(\frac {0.01k}{\alpha_j} \right).
\end{split}
 \end{align*}

  We further make use of the arguments in Section \ref{sect 3.4} to treat the rest of the terms of the right-hand side of \eqref{LM2} to deduce \eqref{sumovermj1}. This completes the proof of Theorem \ref{thmlowerbound1}.
  
\section{Proof of Theorem~\ref{coro:nonvanish}}
  
Theorem~\ref{coro:nonvanish} is obtained by a variant of the proof of Theorem \ref{thmlowerbound1} given in Section \ref{Sec: Pf of Thmlowerbounds}. In fact, note that the bounds in \eqref{LMM} and \eqref{LM} hold for $k=0$.  Thus setting $k=0$ and noting that $\mathcal{M}(\overline{\chi}, 0)=1$, in \eqref{LMM} and \eqref{LM} readily yeild 
\begin{equation} \label{posprop}
 \sumstar_{\chi, q}L \left( \frac{1}{2},\chi \right)  \mathcal{M}(\chi, -1)\Phi\leg{q}{X}  \gg X \quad \mbox{and} \quad
\sumstar_{\chi, q} \left| L \left( \frac{1}{2},\chi \right) \right|^2 |\mathcal{M}(\chi, -1)|^2 \Phi\leg{q}{X} \ll  X.
\end{equation}
  
We remark here that the expression $\mathcal{M}(\chi, -1)$ can now be regarded as a mollifier, similar to those constructed in \cite{LR21} and \cite{DFL21}.  Now we use $\mathcal{N}$ to denote the number of primitive Dirichlet cubic or quartic characters $\chi$ whose conductors does not exceed $X$ and $L(1/2,\chi)\neq 0$.  Using the Cauchy-Schwartz inequality together with the bounds in \eqref{posprop}, we get
\[ X^2 \ll \left( \sumstar_{\chi, q}L \left( \frac{1}{2},\chi \right)  \mathcal{M}(\chi, -1)\Phi\leg{q}{X} \right)^2 \leq \mathcal{N} \sumstar_{\chi, q} \left| L \left( \frac{1}{2},\chi \right) \right|^2 |\mathcal{M}(\chi, -1)|^2 \Phi\leg{q}{X} \ll \mathcal{N} X. \]
Theorem~\ref{coro:nonvanish} follows from the above computation.

\vspace*{.5cm}

\noindent{\bf Acknowledgments.}  P. G. is supported in part by NSFC grant 11871082 and L. Z. by the FRG grant PS43707 and the Faculty Silverstar Award PS65447 at the University of New South Wales (UNSW).  Moreover, the authors would like to thank the anonymous referee for his/her meticulous inspection of the paper and many helpful suggestions, especially on the approach that leads to a positive portion nonvanishing result given in Theorem~\ref{coro:nonvanish}.

\bibliography{biblio}

\begin{bibdiv}
\begin{biblist}

\bib{B&Y}{article}{
      author={Baier, S.},
      author={Young, M.~P.},
       title={Mean values with cubic characters},
        date={2010},
     journal={J. Number Theory},
      volume={130},
      number={4},
       pages={879\ndash 903},
}

\bib{BEW}{book}{
      author={Berndt, B.~C.},
      author={Evans, R.~J.},
      author={Williams, K.~S.},
       title={{G}auss and {J}acobi sums},
      series={Canadian Mathematical Society Series of Monographs and Advanced
  Texts},
   publisher={John Wiley \& Sons},
     address={New York},
        date={1998},
}

\bib{BCR}{article}{
      author={Bettin, S.},
      author={Chandee, V.},
      author={Radziwi{\l \l}, M.},
       title={The mean square of the product of the {R}iemann zeta-function
  with {D}irichlet polynomials},
        date={2017},
     journal={J. Reine Angew. Math.},
      volume={729},
       pages={51\ndash 79},
}

\bib{CFKRS}{article}{
      author={Conrey, J.~B.},
      author={Farmer, D.~W.},
      author={Keating, J.~P.},
      author={Rubinstein, M.~O.},
      author={Snaith, N.~C.},
       title={Integral moments of {$L$}-functions},
        date={2005},
     journal={Proc. London Math. Soc. (3)},
      volume={91},
      number={1},
       pages={33\ndash 104},
}

\bib{DFL21}{article}{
      author={David, C.},
      author={Florea, A.},
      author={Lalin, M.},
       title={Nonvanishing for cubic {$L$}-functions},
        date={2021},
     journal={Forum Math. Sigma},
      volume={9},
       pages={Paper No. e69, 58pp},
}

\bib{DGH}{article}{
      author={Diaconu, A.},
      author={Goldfeld, D.},
      author={Hoffstein, J.},
       title={Multiple {D}irichlet series and moments of zeta and
  {$L$}-functions},
        date={2003},
     journal={Compositio Math.},
      volume={139},
      number={3},
       pages={297\ndash 360},
}

\bib{Gao2021-4}{article}{
      author={Gao, P.},
       title={Bounds for moments of {D}irichlet {$L$}-functions to a fixed
  modulus},
        date={Preprint},
        note={arXiv:2103.00149},
}

\bib{Gao2021-3}{article}{
      author={Gao, P.},
       title={Sharp lower bounds for moments of quadratic {D}irichlet
  {$L$}-functions},
        date={Preprint},
        note={arXiv:2102.04027},
}

\bib{Gao2021-2}{article}{
      author={Gao, P.},
       title={Sharp upper bounds for moments of quadratic {D}irichlet
  {$L$}-functions},
        date={Preprint},
        note={arXiv:2101.08483},
}

\bib{G&Zhao2}{article}{
      author={Gao, P.},
      author={Zhao, L.},
       title={One level density of low-lying zeros of families of
  {$L$}-functions},
        date={2011},
     journal={Compos. Math.},
      volume={147},
      number={1},
       pages={1\ndash 18},
}

\bib{G&Zhao7}{article}{
      author={Gao, P.},
      author={Zhao, L.},
       title={Moments of central values of quartic {D}irichlet
  {$L$}-functions},
        date={2021},
     journal={J. Number Theory},
      volume={228},
       pages={342\ndash 358},
}

\bib{H&L}{article}{
      author={Hardy, G.~H.},
      author={Littlewood, J.~E.},
       title={Contributions to the theory of the {R}iemann zeta-function and
  the theory of the distribution of primes},
        date={1916},
     journal={Acta Math.},
      volume={41},
      number={1},
       pages={119\ndash 196},
}

\bib{Harper}{article}{
      author={Harper, A.~J.},
       title={Sharp conditional bounds for moments of the {R}iemann zeta
  function},
        date={Preprint},
        note={arXiv:1305.4618},
}

\bib{HRS}{article}{
      author={Heap, W.},
      author={Radziwi{\l \l}, M.},
      author={Soundararajan, K.},
       title={Sharp upper bounds for fractional moments of the {R}iemann zeta
  function},
        date={2019},
        ISSN={0033-5606},
     journal={Q. J. Math.},
      volume={70},
      number={4},
       pages={1387\ndash 1396},
}

\bib{H&Sound}{article}{
      author={Heap, W.},
      author={Soundararajan, K.},
       title={Lower bounds for moments of zeta and {$L$}-functions revisited},
        date={Preprint},
        note={arXiv:2007.13154},
}

\bib{H-B81-2}{article}{
      author={Heath-Brown, D.~R.},
       title={Fractional moments of the {R}iemann zeta function},
        date={1981},
     journal={J. London Math. Soc. (2)},
      volume={24},
      number={1},
       pages={65\ndash 78},
}

\bib{H&P}{article}{
      author={Heath-Brown, D.~R.},
      author={Patterson, S.~J.},
       title={The distribution of {K}ummer sums at prime arguments},
        date={1979},
     journal={J. Reine Angew. Math.},
      volume={310},
       pages={111\ndash 130},
}

\bib{In}{article}{
      author={Ingham, A.~E.},
       title={Mean-{V}alue {T}heorems in the {T}heory of the {R}iemann
  {Z}eta-{F}unction},
        date={1927},
     journal={Proc. London Math. Soc. (2)},
      volume={27},
      number={4},
       pages={273\ndash 300},
}

\bib{iwakow}{book}{
      author={Iwaniec, H.},
      author={Kowalski, E.},
       title={{A}nalytic {N}umber {T}heory},
      series={American Mathematical Society Colloquium Publications},
   publisher={American Mathematical Society},
     address={Providence},
        date={2004},
      volume={53},
}

\bib{K&S}{book}{
      author={Katz, N.},
      author={Sarnak, P.},
       title={Random matrices, frobenius eigenvalues, and monodromy},
      series={American Mathematical Society Colloquium Publications},
   publisher={American Mathematical Society},
     address={Providence},
        date={1999},
      volume={45},
}

\bib{Keating-Snaith02}{article}{
      author={Keating, J.~P.},
      author={Snaith, N.~C.},
       title={Random matrix theory and {$L$}-functions at {$s=1/2$}},
        date={2000},
     journal={Comm. Math. Phys.},
      volume={214},
      number={1},
       pages={91\ndash 110},
}

\bib{Kirila}{article}{
      author={Kirila, S.},
       title={An upper bound for discrete moments of the derivative of the
  {R}iemann zeta-function},
        date={2020},
     journal={Mathematika},
      volume={66},
      number={2},
       pages={475\ndash 497},
}

\bib{LR21}{article}{
      author={Lester, S.},
      author={Radziwi\l~\l, M.},
       title={Signs of {F}ourier coefficients of half-integral weight modular
  forms},
        date={2021},
     journal={Math. Ann.},
      volume={379},
      number={3-4},
       pages={1553\ndash 1604},
}

\bib{MVa1}{book}{
      author={Montgomery, H.~L.},
      author={Vaughan, R.~C.},
       title={{M}ultiplicative number theory. {I}. {C}lassical theory},
      series={Cambridge Studies in Advanced Mathematics},
   publisher={Cambridge University Press},
     address={Cambridge},
        date={2007},
      volume={97},
}

\bib{Radziwill}{article}{
      author={Radziwi{\l \l}, M.},
       title={The 4.36th moment of the {R}iemann zeta-function},
        date={2012},
     journal={Int. Math. Res. Not. IMRN},
      number={18},
       pages={4245\ndash 4259},
}

\bib{Radziwill&Sound}{article}{
      author={Radziwi{\l \l}, M.},
      author={Soundararajan, K.},
       title={Moments and distribution of central {$L$}-values of quadratic
  twists of elliptic curves},
        date={2015},
     journal={Invent. Math.},
      volume={202},
      number={3},
       pages={1029\ndash 1068},
}

\bib{Ramachandra1}{article}{
      author={Ramachandra, K.},
       title={Some remarks on the mean value of the {R}iemann zeta function and
  other {D}irichlet series. {I}},
        date={1978},
     journal={Hardy-Ramanujan J.},
      volume={1},
       pages={15pp},
}

\bib{Ramachandra2}{article}{
      author={Ramachandra, K.},
       title={Some remarks on the mean value of the {R}iemann zeta function and
  other {D}irichlet series. {II}},
        date={1980},
     journal={Hardy-Ramanujan J.},
      volume={3},
       pages={1\ndash 24},
}

\bib{Ramachandra3}{article}{
      author={Ramachandra, K.},
       title={Some remarks on the mean value of the {R}iemann zeta function and
  other {D}irichlet series. {III}},
        date={1980},
        ISSN={0066-1953},
     journal={Ann. Acad. Sci. Fenn. Ser. A I Math.},
      volume={5},
      number={1},
       pages={145\ndash 158},
}

\bib{R&Sound}{article}{
      author={Rudnick, Z.},
      author={Soundararajan, K.},
       title={Lower bounds for moments of {$L$}-functions},
        date={2005},
     journal={Proc. Natl. Acad. Sci. USA},
      volume={102},
      number={19},
       pages={6837\ndash 6838},
}

\bib{R&Sound1}{article}{
      author={Rudnick, Z.},
      author={Soundararajan, K.},
       title={Lower bounds for moments of {$L$}-functions: symplectic and
  orthogonal examples},
        date={2006},
     journal={in: Multiple {D}irichlet series, automorphic forms, and analytic
  number theory, 293--303, Proc. Sympos. Pure Math.},
      volume={75},
       pages={Amer. Math. Soc., Providence, RI, 2006.},
}

\bib{Sound2009}{article}{
      author={Soundararajan, K.},
       title={Moments of the {R}iemann zeta function},
        date={2009},
     journal={Ann. of Math. (2)},
      volume={170},
      number={2},
       pages={981\ndash 993},
}

\end{biblist}
\end{bibdiv}
\bibliographystyle{amsxport}

\vspace*{.5cm}

\noindent\begin{tabular}{p{8cm}p{8cm}}
School of Mathematical Sciences & School of Mathematics and Statistics \\
Beihang University & University of New South Wales \\
Beijing 100191 China & Sydney NSW 2052 Australia \\
Email: {\tt penggao@buaa.edu.cn} & Email: {\tt l.zhao@unsw.edu.au} \\
\end{tabular}

\end{document}